\documentclass[preprint,10pt]{elsarticle}

\usepackage[]{caption2}
\usepackage{graphicx}
\usepackage{subfigure}
\usepackage{amsmath}
\usepackage{amssymb}
\usepackage{marvosym}
\usepackage{booktabs}
\usepackage{threeparttable}
\newcommand{\yy}[1]{\hspace*{#1in}}
\usepackage{algorithm}
\usepackage{algpseudocode}
\algloopdefx{RETURN}[1][]{\textbf{return} #1}
\usepackage{color}

\newtheorem{theorem}{Theorem}[section]

\newproof{proof}{\textbf{Proof}}

\usepackage{color}
\usepackage{url}
\usepackage[dvipdfm]{hyperref}

\journal{Theoretical Computer Science}

\begin{document}


\begin{frontmatter}

\title{Subtrees and BC-subtrees of maximum degree $\leq k$ in trees}

\author[YY]{Yu~Yang}
\ead{yangyu@pdsu.edu.cn}
\author[YY]{Xiao-xiao Li}
\ead{lixiaoxiao@pdsu.edu.cn}
\author[YY]{Meng-yuan Jin}
\ead{jinmengyuan@pdsu.edu.cn }
\author[lilong]{Long Li}
\ead{lilong@home.hpu.edu.cn}
\author[Hwang]{Hua Wang}
\ead{hwang@georgiasouthern.edu}
\author[XD]{Xiao-Dong Zhang\corref{cor1}}
\ead{xiaodong@sjtu.edu.cn}

\cortext[cor1]{Corresponding author}
\address[YY]{School of Computer Science, Pingdingshan University, Pingdingshan 467000, China}
\address[lilong]{College of Computer Science and Technology, Henan Polytechnic University, Jiaozuo, 454003, China}
\address[Hwang]{Department of Mathematical Sciences, Georgia Southern University Statesboro, GA 30460, USA}
\address[XD]{School of Mathematical Sciences,Shanghai Jiao Tong University, Shanghai, 200240, China}




\begin{abstract}

The subtrees and BC-subtrees (subtrees where any two leaves are at even distance apart) have been extensively studied in recent years. Such structures, under special constraints on degrees, have applications in many fields. Through an approach based on generating functions, we present recursive algorithms for enumerating various subtrees and BC-subtrees of maximum degree $\leq k$ in trees. The algorithms are illustrated through detailed examples.  We also briefly discuss, in trees, the densities of subtrees (resp.~BC-subtrees) of maximum degree $\leq k(\geq 2)$ among all subtrees (resp.~BC-subtrees).
\end{abstract}


\end{frontmatter}



\section{Introduction}
\label{Section:intro}

Let $G=(V(G), E(G))$ be a simple graph on $n$ vertices and $m$ edges with vertex set $V(G)=\{v_1,v_2,\dots,v_n\}$ and edge set $E(G)=\{e_1,e_2,\dots,$ $e_{m}\}$, the subtree number index, denoted by $ST(G)$, is defined as the total number of non-empty subtrees of $G$. 
Related to subtrees, a BC-subtree is a subtree in which the distances between any two leaves are even. Similar to the subtree number index, we define the BC-subtree number index, denoted by $BST(G)$, as the total number of non-empty BC-subtrees of $G$.
Both of these indices appeared to have applications in the design of reliable communication network \cite{xiao2017trees}, bioinformatics \cite{knudsen2003optimal}, and characterizing structural properties of molecular and graphs \cite{jam1983average,1984Monotonicity,merrifield1989topological,Jamison1990Alternating,haslegrave2014extremal,yang2013bc,yang2015subtrees}.

By way of generating functions, Yan and Yeh presented the algorithms of enumerating subtrees of trees \cite{yan06}. Following similar approaches, Yang et al. \cite{yang2013bc,yang2015algorithms} considered the enumeration of BC-subtrees of trees, unicyclic and edge-disjoint bicyclic graphs. While doing so they introduced two additional distance related variables and weighted cyclic contraction. Also with generating functions, Chin et al. \cite{chin2018subtrees} studied the subtrees of  complete graphs, complete bipartite graphs, theta graphs, as well as the ratio of spanning trees to all subtrees in complete graphs.

Along this line, using ``deletion and contraction", Yang et al. \cite{yang2015subtrees,yang2017spiro} studied the subtree number and BC-subtree number of spiro and polyphenylene chains, molecular graphs of a class of important polycyclic aromatic hydrocarbons that have broad applications in organic and drug synthesis \cite{deng2012wiener,dovslic2010chain,chen2009six,li2012hosoya}. They also further confirmed the reverse correlation between the subtree number and the Winer index on spiro chains, polyphenylene chains, hexagonal chains and phenylene chains \cite{yang2015subtrees,deng2012wiener,yang2017algorithms,gutman1987wiener}. That is, the chains with the minimum subtree numbers coincide with the ones that attain the maximum Wiener indices, and vice versa.

More recently, Yang et al. \cite{yangyu2020} presented the explicit formulae for the expected values of subtree number in random spiro chains and polyphenylene chains, and compared the expected values of these two indices. Through ``path contraction'' that carries weights,  Yang et al. \cite{yang2020trialgorithms}  solved the subtree enumerating problem for tricyclic graphs.

Over the years, extremal problems related to the subtree numbers \cite{jam1983average,haslegrave2014extremal,yang2017spiro, sze05,szekely2007binary, zhang2015minimal, kirk2008largest,zhang2013number,vince2010average} and their relations with other indices such as the Wiener index  \cite{yang2015subtrees,yang2017algorithms,szekely2014extremal,wagner2007correlation,2019Some} have been extensively studied.

On the other hand, it is often interesting to consider substructures with certain fixed parameters. Such questions have applications in bioinformatics \cite{AKUTSU20152}, natural language processing \cite{wangkai20095312583205}, comparison and search of XML data \cite{MilanoSC06}, and logistics \cite{LAU2006385}. In this paper, we will study the numbers of subtrees and BC-subtrees with maximum degree $k$.

First we introduce the necessary terminologies and definitions, and establish some useful lemmas in Section~\ref{Sec:Terminology}. Theoretical background is provided in Section~\ref{sec:theojustifica}. We present the algorithms of enumerating various subtrees and BC-subtrees with maximum degree $\leq k$ of trees in  Section \ref{sec:algoforsubtree} and Section \ref{sec:algoforbcsubtree}, respectively. Section \ref{sec:algoforimplementanddiscuss} illustrates the details of the proposed algorithms. We also discuss the proportion of subtrees (resp.~BC-subtrees) with maximum degree $\leq k(\geq 2)$ in general trees. Lastly, Section \ref{Sec:Conclusion} concludes the paper and proposes directions of potential future work.

\section{Terminology and notations}
\label{Sec:Terminology}
We first introduce the technical notations and lemmas that will be used in the discussion. For more background information one may check \cite{yang2015subtrees,yan06,yangyu2020}.

Let $T =(V(T) ,E(T);f,g)$ be a weighted tree with $V(T)=\{v_1,v_2,\dots,v_n\}$ and $E(T)=\{e_1,e_2,\dots,$ $e_{n-1}\}$.  When the subtree problems are considered, we start with default vertex-weight function $f:=f_1$ and edge-weight function $g:=g_1$, with $f_1:V(T)\rightarrow \Re_0\times \Re_1 \times \cdots \times \Re_{k}$ (here $\Re_i$ represents the $i$-th weight of the vertex) and  $g_1:E(T)\rightarrow \Re$ ($\Re$ is a commutative ring with a unit element 1). Consequently, each vertex weight is a $(k+1)$-dimensional vector $f(v)=(f(v)_0,f(v)_1,\dots,f(v)_k)$ for any $v\in V(T)$. Here $f(v)_i$  represents the number of subtrees rooted at $v$ with maximum degree $\leq k$, with the additional constraint that the degree of $v$ is $i$ $(i=0,1,\dots,k)$. Obviously, for integer $i<0$, $f(v)_i=0$.

When discussing the BC-subtrees, $O_k^v$-subtrees and $E_k^v$-subtrees (to be defined later) of $T$, we start with its vertex-weight function $f:=f_2$ and edge-weight function $g:=g_2$ with $f_2:V(T)\rightarrow (\Re_0\times \Re_1 \times \cdots \times \Re_{k}, \Re_0\times \Re_1 \times \cdots \times \Re_{k})$ and $g_2:E(T)\rightarrow \Re$ (where $\Re$ is a commutative ring with a unit element 1). In this case, each vertex weight is a $(2k+2)$-dimensional vector $f(v)=(f(v)_{o}^{0},f(v)_{o}^{1},\dots ,f(v)_{o}^{k};f(v)_{e}^{0},f(v)_{e}^{1},\dots,f(v)_{e}^{k}$) for any $v\in V(T)$. Here $f(v)_{o}^{i}$ (resp.~$f(v)_{e}^{i}$) represents the number of subtrees rooted at $v$ with all the distance between each pendant vertex and $v$ is odd (resp. even), as well as the maximum degree $\leq k$, and the degree of $v$ is $i$ $(i=0,1,\dots,k)$. We let $f(v)_{o}^{0}=1$ and it is easy to see that, for integer $i<0$, $f(v)_{o}^{i}=0, f(v)_{e}^{i}=0$.

Here we employ the convention that if $\{a_n\}_{\geq 0}$ is a sequence and $j<i$, then $\prod\limits_{t=i}^ja_t=1$ and $\sum\limits_{t=i}^ja_t=0$. For convenience, we list the necessary notations and terminologies below.
\begin{itemize}
  \item $d_T(u,v)$: the distance between vertices $u, v\in V(T)$.
  \item $deg_T(v)$: the degree of $v\in V(T)$.
  \item $``\backslash "$: the removing operation.
  \item $L(T)$: the leaf set of $T$.
   \item $\mathcal{S}_{\leq k}(T)$:  the set of all subtrees with maximum degree $\leq k$ ($\in \mathbb{N}$) of $T$.
  \item $\mathcal{S}_{\leq k}(T;V_S)$:   the set of subtrees containing $V_S(\subseteq V(T))$ with maximum degree $\leq k$.
  \item $\mathcal{S}_{\leq k,o_j}(T;v)$:
  {\small $$ \mathcal{S}_{\leq k,o_j}(T;v)=\{T_1|T_1 \in \mathcal{S}_{\leq k}(T;v)\wedge d_{T_1}(v,l)\equiv1(\text{mod}~2)(\forall~l\in L(T_1))\wedge deg_{T_1}(v)=j\},$$}
      where $j=0,1,\dots,k$, and we call each subtree in $\mathcal{S}_{\leq k,o_j}(T;v)$ the $v_{o_j}^k$-subtree of $T$.
  \item $\mathcal{S}_{\leq k,e_j}(T;v)$:
        {\small$$ \mathcal{S}_{\leq k,e_j}(T;v)=\{T_1|T_1 \in \mathcal{S}_{\leq k}(T;v)\wedge d_{T_1}(v,l)\equiv0(\text{mod}~2)(\forall~l\in L(T_1))\wedge deg_{T_1}(v)=j\}, $$}
       where $j=0,1,\dots,k$, and we call each subtree in $\mathcal{S}_{\leq k,e_j}(T;v)$ the $v_{e_j}^k$-subtree of $T$.
  \item $\omega_{k,o_j}^v(T_1)$, $\omega_{k,e_j}^v(T_1)$: the $\omega_{k,o_j}^v$, $\omega_{k,e_j}^v$ weight of subtree $T_1\in \mathcal{S}_{\leq k}(T;v)$, respectively.

  \item $\mathcal{S}_{{BC}_{\leq k}}(T)$: the set of all BC-subtrees with maximum degree $\leq k$($\geq 2$ is an integer) of $T$.
  \item $\mathcal{S}_{{BC}_{\leq k}}(T;V_S)$:  the set of BC-subtrees containing $V_S(\subseteq V(T))$ with maximum degree $\leq k$($\geq 2$ is an integer).
  \item $\omega_k(T_s)$: the max $k$ degree {\em subtree weight} of $T_s\in \mathcal{S}_{\leq k}(\cdot)$.
  \item $\omega_{bc}^k(T_2)$: the max $k$ degree {\em BC-subtree weight} of $T_2\in \mathcal{S}_{{BC}_{\leq k}}(\cdot)$.
  \item $F_{{BC}_{\leq k}}(\cdot)$: the sum of {\em BC-subtree weight} of BC-subtrees in $\mathcal{S}_{{BC}_{\leq k}}(\cdot)$.
  \item $ \eta_{\leq k}(.)$: the number of subtrees in set $\mathcal{S}_{\leq k}(\cdot)$.
  \item $ \eta_{{BC}_{\leq k}}(.)$: the number of BC-subtrees in set $\mathcal{S}_{{BC}_{\leq k}}(\cdot)$.
\end{itemize}

For a given subtree $T_s\in \mathcal{S}_{\leq k}(T)$, its {\em max $k$ degree} subtree weight is defined as
\[
\omega_k(T_s)=\prod_{v\in V(T_s)}\sum\limits_{i=0}^{k-deg_{T_s}(v)}f(v)_{i}\prod_{e\in E(T_s)}g(e).
\]

And we define the  {\em max $k$ degree} subtree generating function of $T$ by
\[
F_{\leq k}(T; f, g)=\sum_{T_s\in \mathcal{S}_{\leq k}(T)}\omega_k(T_s).
\]
Similarly, the {\em max $k$ degree} subtree generating function of $G$ containing $V_S$ are as follows:
\[
F_{\leq k}(T;f, g;V_S)=\sum_{T_s\in \mathcal{S}_{\leq k}(T_s,V_S)}\omega_k(T_s),
\]

Given $T_1\in \mathcal{S}_{\leq k}(T;v)$($v\in V(T)$), let
$$ S_o(T_1)=\{u|u \in V(T_1)\wedge d_{T_1}(v,u)\equiv1(\text{mod}~2)\} $$
and
$$ S_e(T_1)=\{u|u \in V(T_1)\wedge d_{T_1}(v,u)\equiv0 (\text{mod}~2)\}.$$
Then:
\begin{itemize}
\item the $\omega_{k,o_j}^v$ ($j=0,1,\dots,k$) weight of $T_1$, denoted by $\omega_{k,o_j}^v(T_1)$, is defined as:
\begin{itemize}
  \item If $T_1$ is a weighted single vertex $v$, then $\omega_{k,o_j}^v(T_1)=f(v)_{o}^{j}$;
  \item otherwise,
  $$\omega_{k,o_j}^v(T_1)=a_1a_2a_3a_4a_5,$$
  where $a_1=\prod\limits_{e\in E(T_1)}g(e)$, $a_2=f(v)_{o}^{j-deg_{T_1}(v)}$, $a_3=\prod\limits_{u\in S_o(T_1)}\sum\limits_{i=0}^{k-deg_{T_1}(u)}f(u)_e^{i}$, \\$a_4=\prod\limits_{\substack{u\in S_e(T_1)\backslash v\\ u \notin L(T_1)}}\sum\limits_{i=0}^{k-deg_{T_1}(u)}f(u)_{o}^{i}$, $a_5=\prod\limits_{\substack{u\in S_e(T_1)\backslash v\\ u \in L(T_1)}}\sum\limits_{i=1}^{k-deg_{T_1}(u)}f(u)_{o}^{i}.$

\end{itemize}
\item the $\omega_{k,e_j}^v$ ($j=0,1,\dots,k$) weight of $T_1$, denoted by $\omega_{k,e_j}^v(T_1)$, is defined as:
\begin{itemize}
  \item If $T_1$ is a weighted single vertex $v$, then $\omega_{k,e_j}^v(T_1)=f(v)_{e}^{j}$;
  \item otherwise,
   $$\omega_{k,e_j}^v(T_1)=a_6a_7a_8a_9a_{10}$$
  where $a_6=f(v)_{e}^{j-deg_{T_1}(v)}$, $a_7=\prod\limits_{u\in S_e(T_1)\backslash v}\sum\limits_{i=0}^{k-deg_{T_1}(u)}f(u)_e^{i}$, $a_8=\prod\limits_{\substack{u\in S_o(T_1)\\ u \notin L(T_1)}}\sum\limits_{i=0}^{k-deg_{T_1}(u)}f(u)_{o}^{i}$, $a_9=\prod\limits_{\substack{u\in S_o(T_1)\\ u \in L(T_1)}}\sum\limits_{i=1}^{k-deg_{T_1}(u)}f(u)_{o}^{i}$, $a_{10}=\prod\limits_{e\in E(T_1)}g(e).$
\end{itemize}
\end{itemize}

The $\omega_{k,o_j}^v$, $\omega_{k,e_j}^v$  subtree generating function of $\mathcal{S}_{\leq k}(T;v)$ are respectively defined as
 $$ F_{\leq k,o_j}(T;f,g;v)=\sum \limits_{T_1\in \mathcal{S}_{\leq k}(T;v)}\omega_{k,o_j}^v(T_1)=\sum \limits_{T_1\in \mathcal{S}_{\leq k,o_j}(T;v)}\omega_{k,o_j}^v(T_1), $$
 and
 $$ F_{\leq k,e_j}(T;f,g;v)=\sum \limits_{T_1\in \mathcal{S}_{\leq k}(T;v)}\omega_{k,e_j}^v(T_1)=\sum \limits_{T_1\in \mathcal{S}_{\leq k,e_j}(T;v)}\omega_{k,e_j}^v(T_1).$$

Similarly, for a given BC-subtree $T_2\in \mathcal{S}_{{BC}_{\leq k}}(T)$, we define
$$ B_e(T_2)=\{u|u \in V(T_2)\wedge d_{T_2}(u,v_l)\equiv0 (\text{mod}~2)\} $$
and
 $$ B_o(T_2)=\{u|u \in V(T_2)\wedge d_{T_2}(u,v_l)\equiv1 (\text{mod}~2)\} $$
 where $v_l \in L(T_2)$. The {\em max $k$ degree} BC-subtree weight of $T_2\in \mathcal{S}_{{BC}_{\leq k}}(T)$ is
\begin{equation}
\begin{split}
\omega_{bc}^k(T_2)=(b_1b_2b_3+b_4b_5b_6)b_7
\end{split}
\end{equation}
where
\begin{itemize}
  \item $b_1=\prod\limits_{u\in B_e(T_2)}\sum\limits_{m=0}^{k-deg_{T_2}(u)}f(u)_{e}^{m}$,  $b_2=\prod\limits_{\substack{u\in B_o(T_2)\\u \in L(T_2)}}\sum\limits_{m=1}^{k-deg_{T_2}(u)}f(u)_{o}^{m}$,
  \item $b_3=\prod\limits_{\substack{u\in B_o(T_2)\\u \notin L(T_2)}}\sum\limits_{m=0}^{k-deg_{T_2}(u)}f(u)_{o}^{m}$, $b_4=\prod\limits_{\substack{u\in B_e(T_2)\\u \in L(T_2)}}\sum\limits_{m=1}^{k-deg_{T_2}(u)}f(u)_{o}^{m}$,
  \item $b_5=\prod\limits_{\substack{u\in B_e(T_2)\\u \notin L(T_2)}}\sum\limits_{m=0}^{k-deg_{T_2}(u)}f(u)_{o}^{m}$, $b_6=\prod\limits_{u\in B_o(T_2)}\sum\limits_{m=0}^{k-deg_{T_2}(u)}f(u)_{e}^{m}$,
  \item $b_7=\prod\limits_{e\in E(T_2)}g(e)$.
\end{itemize}

And the  {\em max $k$ degree} BC-subtree generating function of $T$ is
\[
 F_{{BC}_{\leq k}}(T; f, g)=\sum \limits_{T_2\in S_{{BC}_{\leq k}}(T)}\omega_{bc}^k(T_2).
 \]

Similarly,
\[
 F_{{BC}_{\leq k}}(T; f, g;V_S)=\sum \limits_{T_2\in S_{{BC}_{\leq k}}(T;V_S)}\omega_{bc}^k(T_2).
 \]
where $V_S(\subseteq V(T))$.

With above notations, it is not difficult to see
\[
 \eta_{\leq k}(T) = F_{\leq k}(T;(1,0,\dots,0),1),~~~~\eta_{\leq k}(T;V_S)= F_{\leq k}(T; (1,0,\dots,0),1;V_S),
\]
\[
 \eta_{{BC}_{\leq k}}(T) = F_{{BC}_{\leq k}}(T;(1,0,\dots,0;1,0,\dots,0),1),
\]
and
\[
\eta_{{BC}_{\leq k}}(T;V_S)= F_{{BC}_{\leq k}}(T; (1,0,\dots,0;1,0,\dots,0),1;V_S).
\]


\section{Theoretical background}
\label{sec:theojustifica}

Let $T =(V(T),E(T);f,g)$ be a weighted tree on $n\geq 2$ vertices, and let $u$ be a leaf vertex and $p_u =(u, v)$ be a pendant edge of $T$, we define a weighted tree $T' =(V(T'),E(T'); f',g')$ from $T$ with $V(T')=V(T)\backslash u$, $E(T')=E(T)\backslash p_u$,
\begin{align}\label{equ:genertreetrans}
f'(w)_i & =\begin{cases}
f(v)_i+f(v)_{i-1}\sum\limits_{j=0}^{k-1}f(u)_jg(p_u)&\text{if~} w=v,\\
f(w)&\text{otherwise}.
\end{cases}
\end{align}
for any $i=0,1,\dots,k$ and $w\in V(T')$, and $g'(e)=g(e)$ for any $e\in E(T')$.

\begin{theorem}
\label{theo:subtreecondes}
Given $T$ and $T'$ are weighted trees as defined above, we have
\[
F_{\leq k}(T; f, g)=F_{\leq k}(T'; f', g')+\sum\limits_{j=0}^{k}f(u)_j.
\]
\end{theorem}
\begin{proof}
 We can partition the sets $\mathcal{S}_{\leq k}(T)$ and $\mathcal{S}_{\leq k}(T')$ as
\[
\mathcal{S}_{\leq k}(T)=\mathcal{T}_1\cup\mathcal{T}_{1'}\cup\mathcal{T}_2\cup\mathcal{T}_3
\]
and
\[
\mathcal{S}_{\leq k}(T')=\mathcal{T}_1'\cup\mathcal{T}_2'
\]
where
\begin{itemize}
\item $\mathcal{T}_1$ consists of subtrees of $\mathcal{S}_{\leq k}(T)$ that contain the vertex $v$, but not vertex $u$;
\item $\mathcal{T}_{1'}$ consists of subtrees of $\mathcal{S}_{\leq k}(T)$ that contain the edge $p_u =(u, v)$;
\item $\mathcal{T}_2$ consists of subtrees of $\mathcal{S}_{\leq k}(T)$ that contain neither $u$ nor $v$;
\item $\mathcal{T}_3$ consists of subtrees of $\mathcal{S}_{\leq k}(T)$ that contain the vertex $u$, but not $v$;
\item $\mathcal{T}_1'$ consists of subtrees of $\mathcal{S}_{\leq k}(T')$ that contain the vertex $v$;
\item $\mathcal{T}_2'$ consists of subtrees of $\mathcal{S}_{\leq k}(T')$ that do not contain vertex $v$.
\end{itemize}

It is easy to see that:

(i) Both the mapping $m_1:T_1\mapsto T_1'$ between $\mathcal{T}_1$ and $\mathcal{T}_1'$(ignore vertex weight of $v$); and mapping $m_2:T_2\mapsto T_2'$ between $\mathcal{T}_2$ and $\mathcal{T}_2'$ are natural bijections.

(ii) $\mathcal{T}_{1'}=\{T_1+p_u|T_1\in \mathcal{T}_1\}$ where $T_1+p_u$ is the tree obtained from $T_1$ by attaching a pendant edge $p_u=(u,v)$ at vertex $v$ of $T_1$.

(iii) $\mathcal{T}_{3}=\{u\}$.

Note that

\begin{equation}
deg_{T_1'}(v)=deg_{T_1}(v)
\label{equ:addt2pieweight}
\end{equation}

{\small
\begin{equation}
\begin{split}
\sum\limits_{T_1'\in \mathcal{T}'_1}\omega_k(T_1')=&\sum\limits_{T_1'\in \mathcal{T}'_1}\sum\limits_{i=0}^{k-deg_{T_1'}(v)}f'(v)_{i}\frac{\omega_k(T_1')}{\sum\limits_{i=0}^{k-deg_{T_1'}(v)}f'(v)_{i}}\\
&=\sum\limits_{T_1'\in \mathcal{T}'_1}\sum\limits_{i=0}^{k-deg_{T_1'}(v)}\big(f(v)_i+f(v)_{i-1}\sum\limits_{j=0}^{k-1}f(u)_jg(p_u)\big)\frac{\omega_k(T_1')}{\sum\limits_{i=0}^{k-deg_{T_1'}(v)}f'(v)_{i}}
\end{split}
\label{equ:t2pieweight}
\end{equation}}

From (i)-(iii), we have
\begin{equation}
\sum\limits_{T_{1'}\in \mathcal{T}_{1'}}\omega_k(T_{1'})=\sum\limits_{T_1\in \mathcal{T}_1}\frac{\omega_k(T_1)}{\sum\limits_{i=0}^{k-deg_{T_1}(v)}f(v)_{i}}\sum\limits_{i=0}^{k-deg_{T_1}(v)-1}f(v)_{i}\sum\limits_{j=0}^{k-1}f(u)_jg(p_u),
\label{equ:t1piet1weight}
\end{equation}

\begin{equation}
\sum\limits_{T_{2}'\in \mathcal{T}_{2}'}\omega_k(T_{2}')=\sum\limits_{T_{2}\in \mathcal{T}_{2}}\omega_k(T_{2}),
\label{equ:t2piet2weight}
\end{equation}
and
\begin{equation}
\sum\limits_{T_{3}\in \mathcal{T}_{3}}\omega_k(T_{3})=\sum\limits_{j=0}^{k}f(u)_j.
\label{equ:t3piet3weight}
\end{equation}
Immediately following equ.(\ref{equ:t1piet1weight}), we have
{\small
\begin{equation}
\begin{split}
\sum\limits_{T_1\in \mathcal{T}_1}\omega_k(T_1)+\sum\limits_{T_{1'}\in \mathcal{T}_{1'}}\omega_k(T_{1'})
=&\sum\limits_{T_1\in \mathcal{T}_1}\frac{\omega_k(T_1)}{\sum\limits_{i=0}^{k-deg_{T_1}(v)}f(v)_{i}}\Big(\sum\limits_{i=0}^{k-deg_{T_1}(v)}f(v)_{i}\\
&+\sum\limits_{i=0}^{k-deg_{T_1}(v)-1}f(v)_{i}\sum\limits_{j=0}^{k-1}f(u)_jg(p_u)\Big)\\
\end{split}
\label{equ:t2piet2combineweight1st}
\end{equation}}

Furthermore, following (i), the mapping $m_1:T_1\mapsto T_1'$ is a bijection between $\mathcal{T}_1$ and $\mathcal{T}_1'$. We now have
\begin{equation}
\frac{\omega_k(T_1)}{\sum\limits_{i=0}^{k-deg_{T_1}(v)}f(v)_{i}}=\frac{\omega_k(T_1')}{\sum\limits_{i=0}^{k-deg_{T_1'}(v)}f'(v)_{i}}
\label{equ:add8piet3weight}
\end{equation}
 Thus by equs. \eqref{equ:addt2pieweight}, (\ref{equ:t2pieweight}), \eqref{equ:t2piet2combineweight1st} and (\ref{equ:add8piet3weight})
\begin{equation}
\sum\limits_{T_1\in \mathcal{T}_1}\omega_k(T_1)+\sum\limits_{T_{1'}\in \mathcal{T}_{1'}}\omega_k(T_{1'})
=\sum\limits_{T_1'\in \mathcal{T}_1'}\omega_k(T_1') .
\label{equ:t2piet2combineequweight}
\end{equation}

Combining equs. \eqref{equ:t2piet2weight}, \eqref{equ:t3piet3weight}, \eqref{equ:t2piet2combineequweight}, and the definitions of $F_{\leq k}(T; f, g)$ and $F_{\leq k}(T'; f', g')$ we have

\begin{equation}
\begin{split}
F_{\leq k}(T; f, g)
&=\sum\limits_{T_1\in \mathcal{T}_1}\omega_k(T_1)+\sum\limits_{T_{1'}\in \mathcal{T}_{1'}}\omega_k(T_{1'})++\sum\limits_{T_2\in \mathcal{T}_2}\omega_k(T_2)+\sum\limits_{T_3\in \mathcal{T}_3}\omega_k(T_3)\\
&=\sum\limits_{T_1'\in \mathcal{T}'_1}\omega_k(T_1')+\sum\limits_{T_2'\in \mathcal{T}'_2}\omega_k(T_2')+\sum\limits_{T_3\in \mathcal{T}_3}\omega_k(T_3)\\
&=F_{\leq k}(T'; f', g')+\sum\limits_{j=0}^{k}f(u)_j
\end{split}
\label{equ:casefortheodd}
\end{equation}

The theorem thus follows.
\end{proof}

Through similar analysis, we also obtain the following.

\begin{theorem}\label{theo:subtreekcontainon}
Let $T =(V(T),E(T);f,g)$ be a weighted tree on $n\geq 2$ vertices, $u$ is a leaf vertex and $p_u =(u, v)$ is a pendant edge of $T$. Let $T' =(V(T'),E(T'); f',g')$ be a weighted tree defined as above. Then, for arbitrary vertex $v_i\neq u$, we have
$$F_{\leq k}(T;f,g;v_i)=F_{\leq k}(T';f',g';v_i) .$$
\end{theorem}

\begin{theorem}\label{theo:subtreekcontainontwo}
Let $T =(V(T),E(T);f,g)$ be a weighted tree on $n\geq 2$ vertices, assume $u$ is a leaf vertex and $p_u =(u, v)$ is a pendant edge of $T$. Let $T' =(V(T'),E(T'); f',g')$ be a weighted tree defined as above. Then, for arbitrary two distinct vertices $v_i\neq u$, and $v_j\neq u$, we have
$$F_{\leq k}(T;f,g;v_i,v_j)=F_{\leq k}(T';f',g';v_i,v_j) .$$
\end{theorem}

\section{Our algorithms for subtree}
\label{sec:algoforsubtree}
From Theorem \ref{theo:subtreecondes}, we construct Algorithm \ref{Algorithm:subtreenumaxfixedk}
 of computing the generating function $F_{\leq k}(T;f,g)$ of subtrees with maximum degree $\leq k$ of weighted tree $T=(V(T),E(T);f,g)$.

\begin{algorithm}[!ht]
\caption{Generating function $F_{\leq k}(T;f,g)$ for enumerating subtrees with maximum degree $\leq k$ of weighted tree $T=(V(T),E(T);f,g)$}
\label{Algorithm:subtreenumaxfixedk}
\begin{algorithmic}[1]
\State Initialize with $(f(v)_{0},f(v)_{1},\dots ,f(v)_{k}$)($f(v)_{0}=y,f(v)_{1}=,\dots,=f(v)_{k}=0$) for each vertex $v\in V(T)$;

\State Let $T_{tmp}:=T$, and set $N_{T}=0$;\;
\If{$T_{tmp}$ is a single vertex tree $p$}

\State Update $N_T=\sum\limits_{j=0}^{k}f(p)_j$; \;

\Else
\While{$T_{tmp}$ has pendant vertex}
\State Choose a pendant vertex $u$ and let $e=(u,p)$ denote the pendant
\Statex \yy{0.4} edge;  \;

\For{$(i=1;i\leq k;i++)$}

\State Update $f(p)_{i}$ with $f(p)_{i}+f(p)_{i-1} g(e) \sum\limits_{j=0}^{k-1}f(u)_{j}$; \;
\State Update $N_T=N_T+f(u)_{i}$; \;

\EndFor

\State Update $N_T=N_T+f(u)_{0}$; \;
\State Eliminate vertex $u$ and edge $e$ and let $T_{tmp}:=T_{tmp}\backslash((u,p)\cup u)$; \;
\EndWhile
\State Update $N_T=N_T+\sum\limits_{j=0}^{k}f(p)_j$; \;

\EndIf
\State \Return $F(T;f,g)=N_T$.
\end{algorithmic}
\end{algorithm}

Similarly, from Theorem \ref{theo:subtreekcontainon} (resp.~Theorem \ref{theo:subtreekcontainontwo}), we have Algorithm~\ref{Algorithm:subtreenumaxfixedkcontavertex} (resp.~Algorithm~\ref{Algorithm:subtreenumaxfixedkconttwovertex}) of enumerating the subtrees containing a fixed vertex (resp. two distinct vertices) with maximum degree $\leq k$ of a tree.
\begin{algorithm}[!ht]
\caption{Generating function $F_{\leq k}(T;f,g;v_i)$ for enumerating subtrees containing a fixed vertex $v_i\in V(T)$ with maximum degree $\leq k$ of weighted tree $T=(V(T),E(T);f,g)$}
\label{Algorithm:subtreenumaxfixedkcontavertex}
\begin{algorithmic}[1]
\State Initialize with $(f(v)_{0},f(v)_{1},\dots ,f(v)_{k}$)($f(v)_{0}=y,f(v)_{1}=,\dots,=f(v)_{k}=0$) for each vertex $v\in V(T)$;

\State Let $T_{tmp}:=T$, and set $N_{T}=0$;\;
\If{$T_{tmp}$ is the single vertex tree $v_i$}

\State Update $N_T=\sum\limits_{j=0}^{k}f(v_i)_j$; \;

\Else
\While{$T_{tmp}$ has pendant vertex}
\State Choose a pendant vertex $u\neq v_i$ and denote $e=(u,p)$ the pendant
\Statex \yy{0.4} edge;  \;

\For{$(i=1;i\leq k;i++)$}

\State Update $f(p)_{i}$ with $f(p)_{i}+f(p)_{i-1} g(e) \sum\limits_{j=0}^{k-1}f(u)_{j}$; \;

\EndFor

\State Eliminate vertex $u$ and edge $e$ and let $T_{tmp}:=T_{tmp}\backslash((u,p)\cup u)$; \;
\EndWhile
\State Update $N_T=\sum\limits_{j=0}^{k}f(p)_j$; \;

\EndIf
\State \Return $F_{\leq k}(T;f,g;v_i)=N_T$.
\end{algorithmic}
\end{algorithm}

\begin{algorithm}[!ht]
\caption{Generating function $F_{\leq k}(T;f,g;v_i, v_j)$ for enumerating subtrees containing two distinct vertices $v_i, v_j\in V(T)(i\neq j)$ with maximum degree $\leq k$ of weighted tree $T=(V(T),E(T);f,g)$}
\label{Algorithm:subtreenumaxfixedkconttwovertex}
\begin{algorithmic}[1]
\State Initialize with $(f(v)_{0},f(v)_{1},\dots ,f(v)_{k}$)($f(v)_{0}=y,f(v)_{1}=,\dots,=f(v)_{k}=0$) for each vertex $v\in V(T)$;\;

\State Let $T_{tmp}:=T$, and set $N_{T}=0$;\;

\State  \Call{Contract1}{ };\;
\Statex \small{/* denote $P_{v_i v_j}=v_i(u_0)u_1u_2\dots u_{l-1}v_j(u_l)$ the unique path of length $l(\geq1)$ of $T$ connecting $v_i$ and $v_j$, where $v_i=u_0$ and $v_j=u_l$.*/}

\State Update $N_T=\sum\limits_{m=0}^{k-1}f(v_i)_mf(v_j)_m\prod\limits_{i=1}^{l-1}\sum\limits_{m=0}^{k-2}f(u_i)_m\prod\limits_{e\in E(P_{v_i v_j})}g(e)$; \;

\State \Return $F_{\leq k}(T;f,g;v_i, v_j)=N_T$.

\Procedure {Contract1}{ }

\While{$T_{temp}$ has pendant vertex that is different from $v_i$ and $v_j$}
\State Choose a pendant vertex $u$, which is different from $v_i$ and $v_j$, and denote \;

\Statex \yy{0.4} $e=(u,p)$ the pendant edge. \;

\For{$(i=1;i\leq k;i++)$}

\State Update $f(p)_{i}$ with $f(p)_{i}+f(p)_{i-1} g(e) \sum\limits_{j=0}^{k-1}f(u)_{j}$; \;

\EndFor

\State Eliminate vertex $u$ and edge $e$ and let $T_{tmp}:=T_{tmp}\backslash((u,p)\cup u)$; \;

\EndWhile

\EndProcedure

\end{algorithmic}
\end{algorithm}

\section{Our algorithms for BC-subtree}
\label{sec:algoforbcsubtree}
In order to solve the problem of enumerating BC-subtrees with maximum degree $\leq k$ of weighted tree $T=(V(T),E(T);f,g)$, we need firstly to solve the problem of computing $v_{o_j}^k$-subtrees ($j=0,1\dots,k$), $v_{e_j}^k$-subtrees ($j=0,1\dots,k$) of $T$, respectively.

Let $T =(V(T),E(T);f,g)$ be a weighted tree of order $n>1$ rooted at $v_i$ and let $u\neq v_i$ be a pendant vertex of $T$. Suppose $e=(u,v)$ is the pendant edge of $T$. We define a weighted tree $T' =(V(T'),E(T'); f',g')$ of order $n-1$ from $T$ as follows:
$V(T')=V(T)\backslash\{u\}$, $E(T')=E(T)\backslash\{e\}$,
\begin{align*}
f'(v_s)_o^i & =\begin{cases}
f(v)_o^0=1&\text{if } v_s=v \text{ and } i=0\\
f(v)_{o}^{i}+f(v)_{o}^{i-1} g(e) \sum\limits_{m=0}^{k-1}f(u)_{e}^{m}&\text{if } v_s=v\text{ and } 1\leq i\leq k\\
f(v_s)_o^i&\text{otherwise}.
\end{cases}
\end{align*}

\begin{align*}
f'(v_s)_e^i & =\begin{cases}
f(v)_{e}^{i}&\text{ if } v_s=v \text{ and } i=0\\
f(v)_{e}^{i}+f(v)_{e}^{i-1}g(e) \sum\limits_{m=1}^{k-1}f(u)_{o}^{m}&\text{ if } v_s=v \text{ and } 1 \leq i\leq k\\
f(v_s)_e^i&\text{otherwise}.
\end{cases}
\end{align*}

 for any $i=0,1,\dots,k$ and $v_s\in V(T')$ and $g'(e)=g(e)$ for any $e\in E(T')$.
\begin{theorem}
With the above notations, we have
\[
F_{\leq k,o_j}(T;f,g;v_{i})=F_{\leq k,o_j}(T'; f', g';v_i);
\]
\[
F_{\leq k,e_j}(T;f,g;v_{i})=F_{\leq k,e_j}(T'; f', g';v_i).
\]
\label{theorem:oddeventransbcmaxk}
\end{theorem}
\begin{proof}
We consider two cases.
\begin{itemize}
\item If $d_T(v_i,v)$ is odd, we partition the sets $\mathcal{S}_{\leq k}(T;v_i)$ and $\mathcal{S}_{\leq k}(T';v_i)$ as
\[
\mathcal{S}_{\leq k}(T;v_i)=\mathcal{T}_1\cup\mathcal{T}_{1'}\cup\mathcal{T}_2\cup\mathcal{T}_3
\]
and
\[
\mathcal{S}_{\leq k}(T';v_i)=\mathcal{T}_1'\cup\mathcal{T}_2'
\]
where
\begin{itemize}
\item $\mathcal{T}_1$ consists of subtrees of $\mathcal{S}_{\leq k}(T;v_i)$ that contain the vertex $v$, but not vertex $u$;
\item $\mathcal{T}_{1'}$ consists of subtrees of $\mathcal{S}_{\leq k}(T;v_i)$ that contain the edge $p_u =(u, v)$;
\item $\mathcal{T}_2$ consists of subtrees of $\mathcal{S}_{\leq k}(T;v_i)$ that contain neither $u$ nor $v$;
\item $\mathcal{T}_3$ consists of subtrees of $\mathcal{S}_{\leq k}(T;v_i)$ that contain the vertex $u$, but not $v$;
\item $\mathcal{T}_1'$ consists of subtrees of $\mathcal{S}_{\leq k}(T';v_i)$ that contain the vertex $v$;
\item $\mathcal{T}_2'$ consists of subtrees of $\mathcal{S}_{\leq k}(T';v_i)$ that do not contain vertex $v$.
\end{itemize}

Similar to before, we claim that

(i) Both the mapping $m_1:T_1\mapsto T_1'$ between $\mathcal{T}_1$ and $\mathcal{T}_1'$(ignore vertex weight of $v$); and mapping $m_2:T_2\mapsto T_2'$ between $\mathcal{T}_2$ and $\mathcal{T}_2'$ are natural bijections.

(ii) $\mathcal{T}_{1'}=\{T_1+p_u|T_1\in \mathcal{T}_1\}$ where $T_1+p_u$ is the tree obtained from $T_1$ by attaching a pendant edge $p_u=(u,v)$ at vertex $v$ of $T_1$.

(iii) $\mathcal{T}_{3}=\emptyset$.

Note that

\begin{equation}
deg_{T_1'}(v)=deg_{T_1}(v)
\label{equ:bcaddt2pieweight}
\end{equation}

{\small
\begin{equation}
\begin{split}
\sum\limits_{T_1'\in \mathcal{T}'_1}\omega_{k,o_j}^v(T_1')=&\sum\limits_{T_1'\in \mathcal{T}'_1}\sum\limits_{i=0}^{k-deg_{T_1'}(v)}f'(v)_e^{i}\frac{\omega_{k,o_j}^v(T_1')}{\sum\limits_{i=0}^{k-deg_{T_1'}(v)}f'(v)_{e}^{i}}\\
&=\sum\limits_{T_1'\in \mathcal{T}'_1}\sum\limits_{i=0}^{k-deg_{T_1'}(v)}\big(f(v)_e^i+f(v)_e^{i-1}\sum\limits_{j=1}^{k-1}f(u)_0^jg(p_u)\big)\frac{\omega_{k,o_j}^v(T_1')}{\sum\limits_{i=0}^{k-deg_{T_1'}(v)}f'(v)_e^{i}}
\end{split}
\label{equ:bct2pieweight}
\end{equation}}
By (i)-(iii), we have
\begin{equation}
\sum\limits_{T_{1'}\in \mathcal{T}_{1'}}\omega_{k,o_j}^v(T_{1'})=\sum\limits_{T_1\in \mathcal{T}_1}\frac{\omega_{k,o_j}^v(T_1)}{\sum\limits_{i=0}^{k-deg_{T_1}(v)}f(v)_e^{i}}\sum\limits_{i=0}^{k-deg_{T_1}(v)-1}f(v)_e^{i}\sum\limits_{j=1}^{k-1}f(u)_o^jg(p_u),
\label{equ:bct1piet1weight}
\end{equation}

\begin{equation}
\sum\limits_{T_{2}'\in \mathcal{T}_{2}'}\omega_{k,o_j}^v(T_{2}')=\sum\limits_{T_{2}\in \mathcal{T}_{2}}\omega_{k,o_j}^v(T_{2}),
\label{equ:bct2piet2weight}
\end{equation}
and
\begin{equation}
\sum\limits_{T_{3}\in \mathcal{T}_{3}}\omega_{k,o_j}^v(T_{3})=0.
\label{equ:bct3piet3weight}
\end{equation}
Immediately following (\ref{equ:bct1piet1weight}), we have
{\small
\begin{equation}
\begin{split}
\sum\limits_{T_1\in \mathcal{T}_1}\omega_{k,o_j}^v(T_1)+\sum\limits_{T_{1'}\in \mathcal{T}_{1'}}\omega_{k,o_j}^v(T_{1'})
=&\sum\limits_{T_1\in \mathcal{T}_1}\frac{\omega_{k,o_j}^v(T_1)}{\sum\limits_{i=0}^{k-deg_{T_1}(v)}f(v)_e^{i}}\Big(\sum\limits_{i=0}^{k-deg_{T_1}(v)}f(v)_e^{i}\\
&+\sum\limits_{i=0}^{k-deg_{T_1}(v)-1}f(v)_e^{i}\sum\limits_{j=1}^{k-1}f(u)_o^jg(p_u)\Big)
\end{split}
\label{equ:t2piet2combineweight}
\end{equation}}

Furthermore, following (i), the mapping $m_1:T_1\mapsto T_1'$ is a bijection between $\mathcal{T}_1$ and $\mathcal{T}_1'$. We now have
$$ \frac{\omega_{k,o_j}^v(T_1)}{\sum\limits_{i=0}^{k-deg_{T_1}(v)}f(v)_e^{i}}=\frac{\omega_{k,o_j}^v(T_1')}{\sum\limits_{i=0}^{k-deg_{T_1'}(v)}f'(v)_e^{i}} $$
Thus by (\ref{equ:bct2pieweight}) and (\ref{equ:t2piet2combineweight}),
\begin{equation}
\sum\limits_{T_1\in \mathcal{T}_1}\omega_{k,o_j}^v(T_1)+\sum\limits_{T_{1'}\in \mathcal{T}_{1'}}\omega_{k,o_j}^v(T_{1'})
=\sum\limits_{T_1'\in \mathcal{T}'_1}\omega_{k,o_j}^v(T_1').
\label{equ:bct2piet2combineequweight}
\end{equation}

From (\ref{equ:bct2piet2weight}), (\ref{equ:bct3piet3weight}), (\ref{equ:bct2piet2combineequweight}) and according to the definitions of $F_{\leq k,o_j}(T;f,g;v_{i})$ and $F_{\leq k,o_j}(T'; f', g';v_i)$ we have

{\small \begin{equation}
\begin{split}
F_{\leq k,o_j}(T;f,g;v_{i})\nonumber
&=\sum\limits_{T_1\in \mathcal{T}_1}\omega_{k,o_j}^v(T_1)
+\sum\limits_{T_{1'}\in \mathcal{T}_{1'}}\omega_{k,o_j}^v(T_{1'})+\sum\limits_{T_2\in \mathcal{T}_2}\omega_{k,o_j}^v(T_2)+\sum\limits_{T_3\in \mathcal{T}_3}\omega_{k,o_j}^v(T_3)\nonumber\\
&=\sum\limits_{T_1'\in \mathcal{T}'_1}\omega_{k,o_j}^v(T_1')+\sum\limits_{T_2'\in \mathcal{T}'_2}\omega_{k,o_j}^v(T_2')\nonumber\\
&=F_{\leq k,o_j}(T'; f', g';v_i).\nonumber
\end{split}
\label{equ:casefortheodd}
\end{equation}}

\item If $d_T(v_i,v)$ is even, we partition the sets $\mathcal{S}_{\leq k}(T;v_i)$ and $\mathcal{S}_{\leq k}(T';v_i)$ as
\[
\mathcal{S}_{\leq k}(T;v_i)=\mathcal{T}_{1,1}\cup\mathcal{T}_{1,2}\cup\mathcal{T}_{1'}\cup\mathcal{T}_2\cup\mathcal{T}_3
\]
and
\[
\mathcal{S}_{\leq k}(T';v_i)=\mathcal{T}_{1,1}'\cup\mathcal{T}_{1,2}'\cup\mathcal{T}_2'
\]
where
\begin{itemize}
\item $\mathcal{T}_{1,1}$ consists of subtrees of $\mathcal{S}_{\leq k}(T;v_i)$ that contain the leaf vertex $v$, but not vertex $u$;
\item $\mathcal{T}_{1,2}$ consists of subtrees of $\mathcal{S}_{\leq k}(T;v_i)$ that contain the nonleaf vertex $v$, but not vertex $u$;
\item $\mathcal{T}_{1'}$ consists of subtrees of $\mathcal{S}_{\leq k}(T;v_i)$ that contain the edge $p_u =(u, v)$;
\item $\mathcal{T}_2$ consists of subtrees of $\mathcal{S}_{\leq k}(T;v_i)$ that contain neither $u$ nor $v$;
\item $\mathcal{T}_3$ consists of subtrees of $\mathcal{S}_{\leq k}(T;v_i)$ that contain the vertex $u$, but not $v$;
\item $\mathcal{T}_{1,1}'$ consists of subtrees of $\mathcal{S}_{\leq k}(T';v_i)$ that contain the leaf vertex $v$;
\item $\mathcal{T}_{1,2}'$ consists of subtrees of $\mathcal{S}_{\leq k}(T';v_i)$ that contain the nonleaf vertex $v$;
\item $\mathcal{T}_2'$ consists of subtrees of $\mathcal{S}_{\leq k}(T';v_i)$ that do not contain vertex $v$.
\end{itemize}

Similarly, we have

(i) The mapping $m_1:T_{1,1}\mapsto T_{1,1}'$ between $\mathcal{T}_{1,1}$ and $\mathcal{T}_{1,1}'$(ignore vertex weight of $v$); the mapping $m_2:T_{1,2}\mapsto T_{1,2}'$ between $\mathcal{T}_{1,2}$ and $\mathcal{T}_{1,2}'$(ignore vertex weight of $v$), and mapping $m_3:T_2\mapsto T_2'$ between $\mathcal{T}_2$ and $\mathcal{T}_2'$ are natural bijections.

(ii) $\mathcal{T}_{1'}=\{T_{1,1}+p_u|T_{1,1}\in \mathcal{T}_{1,1}\}\cup \{T_{1,2}+p_u|T_{1,2}\in \mathcal{T}_{1,2}\}$ where $T_{1,1}+p_u$ (resp. $T_{1,2}+p_u$) is the tree obtained from $T_1$ by attaching a pendant edge $p_u=(u,v)$ at vertex $v$ of $T_{1,1}$ (resp. $T_{1,2}$).

(iii) $\mathcal{T}_{3}=\emptyset$.

Note that
\begin{equation}
deg_{T_{1,1}'}(v)=deg_{T_{1,1}}(v)
\label{equ:bcevenaddt2pieweight1}
\end{equation}
\begin{equation}
deg_{T_{1,2}'}(v)=deg_{T_{1,2}}(v)
\label{equ:bcevenaddt2pieweight2}
\end{equation}

{\small
\begin{equation}
\begin{split}
\sum\limits_{T_{1,1}'\in \mathcal{T}'_{1,1}}\omega_{k,o_j}^v(T_{1,1}')=&\sum\limits_{T_{1,1}'\in \mathcal{T}'_{1,1}}\sum\limits_{i=1}^{k-deg_{T_{1,1}'}(v)}f'(v)_o^{i}\frac{\omega_{k,o_j}^v(T_{1,1}')}{\sum\limits_{i=1}^{k-deg_{T_{1,1}'}(v)}f'(v)_{o}^{i}}\\
=&\sum\limits_{T_{1,1}'\in \mathcal{T}'_{1,1}}\sum\limits_{i=1}^{k-deg_{T_{1,1}'}(v)}\big(f(v)_o^i+f(v)_o^{i-1}\sum\limits_{j=0}^{k-1}f(u)_0^jg(p_u)\big)\frac{\omega_{k,o_j}^v(T_{1,1}')}{\sum\limits_{i=1}^{k-deg_{T_{1,1}'}(v)}f'(v)_o^{i}}
\end{split}
\label{equ:bct2pieweighteven}
\end{equation}}

{\small
\begin{equation}
\begin{split}
\sum\limits_{T_{1,2}'\in \mathcal{T}'_{1,2}}\omega_{k,o_j}^v(T_{1,2}')=&\sum\limits_{T_{1,2}'\in \mathcal{T}'_{1,2}}\sum\limits_{i=0}^{k-deg_{T_{1,2}'}(v)}f'(v)_o^{i}\frac{\omega_{k,o_j}^v(T_{1,2}')}{\sum\limits_{i=0}^{k-deg_{T_{1,2}'}(v)}f'(v)_{o}^{i}}\\
=&\sum\limits_{T_{1,2}'\in \mathcal{T}'_{1,2}}\sum\limits_{i=0}^{k-deg_{T_{1,2}'}(v)}\big(f(v)_o^i+f(v)_o^{i-1}\sum\limits_{j=0}^{k-1}f(u)_0^jg(p_u)\big)\frac{\omega_{k,o_j}^v(T_{1,2}')}{\sum\limits_{i=0}^{k-deg_{T_{1,2}'}(v)}f'(v)_o^{i}}
\end{split}
\label{equ:bct2pieweighteven2}
\end{equation}}

By (i)-(iii), we have
\begin{equation}
\begin{split}
&\sum\limits_{T_{1,1}\in \mathcal{T}_{1,1}}\omega_{k,o_j}^v(T_{1,1})+\sum\limits_{T_{1,2}\in \mathcal{T}_{1,2}}\omega_{k,o_j}^v(T_{1,2})+\sum\limits_{T_{1'}\in \mathcal{T}_{1'}}\omega_{k,o_j}^v(T_{1'})\\
&=\sum\limits_{T_{1,1}\in \mathcal{T}_{1,1}}\frac{\omega_{k,o_j}^v(T_{1,1})}{\sum\limits_{i=1}^{k-deg_{T_{1,1}}(v)}f(v)_o^{i}}\bigg[\sum\limits_{i=1}^{k-deg_{T_{1,1}}(v)}f(v)_o^{i}+\sum\limits_{i=1}^{k-deg_{T_{1,1}}(v)-1}f(v)_o^{i}\sum\limits_{j=0}^{k-1}f(u)_e^jg(p_u)\bigg],
\\
&+\sum\limits_{T_{1,2}\in \mathcal{T}_{1,2}}\frac{\omega_{k,o_j}^v(T_{1,2})}{\sum\limits_{i=0}^{k-deg_{T_{1,2}}(v)}f(v)_o^{i}}\bigg[\sum\limits_{i=0}^{k-deg_{T_{1,2}}(v)}f(v)_o^{i}+\sum\limits_{i=0}^{k-deg_{T_{1,2}}(v)-1}f(v)_o^{i}\sum\limits_{j=0}^{k-1}f(u)_e^jg(p_u)\bigg],
\end{split}
\label{equ:bct1piet1weighteven222}
\end{equation}

\begin{equation}
\sum\limits_{T_{2}'\in \mathcal{T}_{2}'}\omega_{k,o_j}^v(T_{2}')=\sum\limits_{T_{2}\in \mathcal{T}_{2}}\omega_{k,o_j}^v(T_{2}),
\label{equ:bct2piet2weighteven222}
\end{equation}
and
\begin{equation}
\sum\limits_{T_{3}\in \mathcal{T}_{3}}\omega_{k,o_j}^v(T_{3})=0.
\label{equ:bct3piet3weighteven222}
\end{equation}

Similarly, from (i), with equs. \eqref{equ:bcevenaddt2pieweight1},  \eqref{equ:bcevenaddt2pieweight2}, we have
\begin{equation}
\frac{\omega_{k,o_j}^v(T_{1,1}')}{\sum\limits_{i=1}^{k-deg_{T_{1,1}'}(v)}f'(v)_o^{i}}=\frac{\omega_{k,o_j}^v(T_{1,1})}{\sum\limits_{i=1}^{k-deg_{T_{1,1}}(v)}f(v)_o^{i}}
\label{equ:add8piet3weighteven}
\end{equation}
and
\begin{equation}
\frac{\omega_{k,o_j}^v(T_{1,2}')}{\sum\limits_{i=1}^{k-deg_{T_{1,2}'}(v)}f'(v)_o^{i}}=\frac{\omega_{k,o_j}^v(T_{1,2})}{\sum\limits_{i=1}^{k-deg_{T_{1,2}}(v)}f(v)_o^{i}}
\label{equ:add8piet3weighteven2}
\end{equation}
Thus by equs. (\ref{equ:bct2pieweighteven}), \eqref{equ:bct2pieweighteven2}, (\ref{equ:bct1piet1weighteven222}), \eqref{equ:add8piet3weighteven} and \eqref{equ:add8piet3weighteven2}, we have
\begin{equation}
\begin{split}
\sum\limits_{T_{1,1}\in \mathcal{T}_{1,1}}\omega_{k,o_j}^v(T_{1,1})+\sum\limits_{T_{1,2}\in \mathcal{T}_{1,2}}\omega_{k,o_j}^v(T_{1,2})+\sum\limits_{T_{1'}\in \mathcal{T}_{1'}}\omega_{k,o_j}^v(T_{1'})
\\=\sum\limits_{T_{1,1}'\in \mathcal{T}'_{1,1}}\omega_{k,o_j}^v(T_{1,1}')+\sum\limits_{T_{1,2}'\in \mathcal{T}'_{1,2}}\omega_{k,o_j}^v(T_{1,2}').
\end{split}
\label{equ:bct2piet2combineequweighteven33}
\end{equation}

Similarly, from (\ref{equ:bct2piet2weighteven222}), (\ref{equ:bct3piet3weighteven222}), (\ref{equ:bct2piet2combineequweighteven33}) and according to the definitions of $F_{\leq k,o_j}(T;f,g;v_{i})$ and $F_{\leq k,o_j}(T'; f', g';v_i)$ we have
{\small \begin{equation}
\begin{split}
F_{\leq k,o_j}(T;f,g;v_{i})\nonumber
=&\sum\limits_{T_{1,1}\in \mathcal{T}_{1,1}}\omega_{k,o_j}^v(T_{1,1})
\sum\limits_{T_{1,2}\in \mathcal{T}_{1,2}}\omega_{k,o_j}^v(T_{1,2})+\sum\limits_{T_{1'}\in \mathcal{T}_{1'}}\omega_{k,o_j}^v(T_{1'})+\sum\limits_{T_2\in \mathcal{T}_2}\omega_{k,o_j}^v(T_2)\\
&+\sum\limits_{T_3\in \mathcal{T}_3}\omega_{k,o_j}^v(T_3)\nonumber\\
&=\sum\limits_{T_{1,1}'\in \mathcal{T}'_{1,1}}\omega_{k,o_j}^v(T_{1,1}')+\sum\limits_{T_{1,2}'\in \mathcal{T}'_{1,2}}\omega_{k,o_j}^v(T_{1,2}')+\sum\limits_{T_2'\in \mathcal{T}'_2}\omega_{k,o_j}^v(T_2')\nonumber\\
&=F_{\leq k,o_j}(T'; f', g';v_i).\nonumber
\end{split}
\label{equ:casefortheeven}
\end{equation}}
\end{itemize}
We skip the technical details for the case of $F_{\leq k,e_j}(T;f,g;v_{i})=F_{\leq k,e_j}(T'; f', g';v_i)$.

The theorem then follows.

\end{proof}

From Theorem \ref{theorem:oddeventransbcmaxk}, we have the Algorithm \ref{Algorithm:degreecontibcsubtreecontfixedu} of enumerating $v_{o_j}^k$-subtrees ($j=0,1\dots,k$), $v_{e_j}^k$-subtrees ($j=0,1\dots,k$) of $T$, respectively.

\begin{algorithm}[!ht]
\caption{Generating function $F_{\leq k,o_j}(T;f,g;v_{i}),F_{\leq k,e_j}(T;f,g;v_{i})$ ($j=0,1\dots,k$) for a fixed vertex $v_{i}$.}
\label{Algorithm:degreecontibcsubtreecontfixedu}
\begin{algorithmic}[1]
\State Initialize with $(f(v_s)_{o}^{0},f(v_s)_{o}^{1},\dots ,f(v_s)_{o}^{k};f(v_s)_{e}^{0},f(v_s)_{e}^{1},\dots,f(v_s)_{e}^{k}$)($f(v_s)_{o}^{0}=1, f(v_s)_{e}^{0}=y,f(v_s)_{o}^{1}=,\dots,=f(v_s)_{o}^{k}=f(v_s)_{e}^{1}=,\dots,=f(v_s)_{e}^{k}=0$) for each vertex $v_s \in V(T)$;

\If{$v_i$ is a single vertex tree}

\State Set $p:=v_i$; \;

\Else
\State \Call{Contract2}{ };\;

\EndIf

\For{$(j=0;j\leq k;j++)$}

\State \Return $F_{\leq k,o_j}(T;f,g;v_{i})=f(p)_{o}^{j}$ and $F_{\leq k,e_j}(T;f,g;v_{i})=f(p)_{e}^{j}$;
\EndFor

\Procedure {Contract2}{ }

\While{$T$ is not a single vertex tree}
\State Choose a pendant vertex $u'\neq v_i$ and denote $e=(u',p)$ the pendant
\Statex \yy{0.4} edge;  \;
\For{$(j=1;j\leq k;j++)$}

\State  Update $f(p)_{o}^{j}$ with $f(p)_{o}^{j}+f(p)_{o}^{j-1} g(e) \sum\limits_{m=0}^{k-1}f(u')_{e}^{m}$;\;

\State Update $f(p)_{e}^{j}$ with $f(p)_{e}^{j}+f(p)_{e}^{j-1}g(e) \sum\limits_{m=1}^{k-1}f(u')_{o}^{m}$;\;
\EndFor
\State Eliminate vertex $u'$ and edge $e$ and let $T:=T\backslash((u',p)\cup u')$;\;

\EndWhile

\EndProcedure

\end{algorithmic}
\end{algorithm}

Following the same conditions and notations, essentially the same argument yields the following theorems:
\begin{theorem} Given two distinct vertices $v_i$ and $v_j$, say $v_i\neq u$, $v_j\neq u$, we have
\[
F_{{BC}_{\leq k}}(T; f, g;v_i,v_j)=F_{{BC}_{\leq k}}(T'; f', g';v_i,v_j).
\]
\label{theorem:bcnumcontainfixedtwovertextheor}
\end{theorem}

Assume $e=(u,v)$ is an edge of $T =(V(T) ,E(T);f,g)$, denote by $T_u$ (resp. $T_v$) the tree of $T\backslash e$ that contains $u$ (resp. $v$),  the generating function $F_{{BC}_{\leq k}}(T; f, g)$ and $F_{{BC}_{\leq k}}(T; f, g;v)$ follows.

\begin{theorem}
\begin{align*}
F_{{BC}_{\leq k}}(T; f, g) = & F_{{BC}_{\leq k}}(T_u; f, g)+F_{{BC}_{\leq k}}(T_v; f, g) \\
&+\sum\limits_{i=1}^{k-1}F_{\leq k,o_i}(T_v; f, g; v) \times \sum\limits_{i=0}^{k-1}F_{\leq k,e_i}(T_u; f, g; u)g(e) \\
&+\sum\limits_{i=0}^{k-1}F_{\leq k,e_i}(T_v; f, g;v) \times \sum\limits_{i=1}^{k-1}F_{\leq k,o_i}(T_u; f, g; u)g(e).
\end{align*}
\label{theorem:bcnumcomputingtheor}
\end{theorem}

\begin{theorem}
\begin{align*}
F_{{BC}_{\leq k}}(T; f, g;v) = & F_{{BC}_{\leq k}}(T_v; f, g;v) +\sum\limits_{i=1}^{k-1}F_{\leq k,o_i}(T_v; f, g; v) \times \sum\limits_{i=0}^{k-1}F_{\leq k,e_i}(T_u; f, g; u)g(e) \\
&+\sum\limits_{i=0}^{k-1}F_{\leq k,e_i}(T_v; f, g;v) \times \sum\limits_{i=1}^{k-1}F_{\leq k,o_i}(T_u; f, g; u)g(e).
\end{align*}
\label{theorem:bcnumcontainvertextheor}
\end{theorem}

The following identity will be used in Algorithm \ref{Algorithm:degereekbcsubtreenum}
\begin{equation}\label{equ:ntlongformule}
\begin{split}
N_T=&N_T+\Big(\sum\limits_{i=1}^{k-1}F_{\leq k,o_i}(T_p; f, g; p) \times \sum\limits_{i=0}^{k-1}F_{\leq k,e_i}(T_u; f, g; u) \\
&+ \sum\limits_{i=0}^{k-1}F_{\leq k,e_i}(T_p; f, g; p) \times \sum\limits_{i=1}^{k-1}F_{\leq k,o_i}(T_u; f, g; u)\Big)g(e);
\end{split}
\end{equation}

\begin{algorithm}[!ht]
\caption{Generating function $F_{{BC}_{\leq k}}(T;f,g)$ for enumerating BC-subtrees with maximum degree $\leq k$ of weighted tree $T=(V(T),E(T);f,g)$}
\label{Algorithm:degereekbcsubtreenum}
\begin{algorithmic}[1]
\State Initialize with $(f(v)_{o}^{0},f(v)_{o}^{1},\dots ,f(v)_{o}^{k};f(v)_{e}^{0},f(v)_{e}^{1},\dots,f(v)_{e}^{k}$)($f(v_s)_{o}^{0}=1, f(v_s)_{e}^{0}=y,f(v_s)_{o}^{1}=,\dots,=f(v_s)_{o}^{k}=f(v_s)_{e}^{1}=,\dots,=f(v_s)_{e}^{k}=0$) for each vertex $v\in V(T)$. Define $N_T=0$.

\If{$T$ is not a single vertex tree}

\State \Call{Contract3}{$T,f,g$};\;

\EndIf

\State \Return $F_{{BC}_{\leq k}}(T;f,g)=N_T$.

\Procedure {Contract3}{$T,f,g$}

\While{$T$ has at least one edge}

\State Choose a random edge $e\in E(T)$ and denote $e=(u,p)$;  \;

\State   Eliminate the edge $(u,p)$, and denote $T_u$ $(T_p)$ the tree of $T \backslash (u,p)$
\Statex \yy{0.4}that contains $u$ ($p$);\;

\State Calculate $F_{\leq k,o_j}(T_u; f, g; u)$, $F_{\leq k,e_j}(T_u; f, g; u)$ (resp.
\Statex \yy{0.4} $F_{\leq k,o_j}(T_p; f, g; p)$, $F_{\leq k,e_j}(T_p; f, g; p))$ ($j=0,1\dots,k$) by  setting
 \Statex \yy{0.4} $ T:=T_u; v_i:=u (resp.~T:=T_p; v_i:=p)$ and calling
 Algorithm~\ref{Algorithm:degreecontibcsubtreecontfixedu};

\State Update $N_T$ with Eq.~\eqref{equ:ntlongformule} \;

\State Calculate $F_{{BC}_{\leq k}}(T_u;f,g)$ (resp.~$F_{{BC}_{\leq k}}(T_p;f,g)$) by setting $T:=T_u$
\Statex \yy{0.4}(resp.~$T:=T_p$), $f:=f, g:=g$ and calling procedure
 \Statex \yy{0.4} \Call{Contract3}{$T,f,g$} recursively;\;
\EndWhile

\EndProcedure

\end{algorithmic}
\end{algorithm}

We also obtain Algorithm~\ref{Algorithm:BCsubtreenumaxfixedkcontavertex} (resp.~Algorithm~\ref{Algorithm:BCsubtreenumaxfixedkconttwovertex}) of enumerating the BC-subtrees containing a fixed vertex (resp. two fixed vertices) with maximum degree $\leq k$ of a tree.

\begin{equation}\label{equ:bcntlongformule}
\begin{split}
N_T=&N_T+\Big(\sum\limits_{i=1}^{k-1}F_{\leq k,o_i}(T_{v_{t}}; f, g; v_{t}) \times \sum\limits_{i=0}^{k-1}F_{\leq k,e_i}(T_{v_{tmp}}; f, g; v_{tmp}) \\
&+ \sum\limits_{i=0}^{k-1}F_{\leq k,e_i}(T_{v_{t}}; f, g; v_{t}) \times \sum\limits_{i=1}^{k-1}F_{\leq k,o_i}(T_{v_{tmp}}; f, g; v_{tmp})\Big)g(e);
\end{split}
\end{equation}

\begin{algorithm}[!ht]
\caption{Generating function $F_{{BC}_{\leq k}}(T;f,g;v_t)$ for enumerating BC-subtrees containing a fixed vertex $v_t\in V(T)$ with maximum degree $\leq k$ of weighted tree $T=(V(T),E(T);f,g)$}
\label{Algorithm:BCsubtreenumaxfixedkcontavertex}
\begin{algorithmic}[1]
\State Initialize with $(f(v_s)_{o}^{0},f(v_s)_{o}^{1},\dots ,f(v_s)_{o}^{k};f(v_s)_{e}^{0},f(v_s)_{e}^{1},\dots,f(v_s)_{e}^{k}$)($f(v_s)_{o}^{0}=1, f(v_s)_{e}^{0}=y,f(v_s)_{o}^{1}=,\dots,=f(v_s)_{o}^{k}=f(v_s)_{e}^{1}=,\dots,=f(v_s)_{e}^{k}=0$) for each vertex $v_s \in V(T)$,
and set $N_{T}=0$;\;
\If{$T$ is not the single vertex tree $v_t$}

\State \Call{Contract4}{$T,f,g$};\;

\EndIf

\State \Return $F_{{BC}_{\leq k}}(T;f,g;v_t)=N_T$.

\Procedure {Contract4}{$T,f,g$}

\While{$v_{t}$ has neighbour vertex}
\State Choose a neighbour vertex $v_{tmp}$ of $v_{t}$  and denote edge $e=(v_{tmp},v_{t})$;  \;

\State   Eliminate the edge $e$, and denote $T_{v_{t}}$ $(T_{v_{tmp}})$ the tree of $T \backslash (v_{tmp},v_{t})$
\Statex \yy{0.4}that contains $v_{t}$ ($v_{tmp}$);\;

\State Calculate $F_{\leq k,o_j}(T_{v_{t}}; f, g; v_{t})$, $F_{\leq k,e_j}(T_{v_{t}}; f, g; v_{t})$ (resp.
\Statex \yy{0.4} $F_{\leq k,o_j}(T_{v_{tmp}}; f, g; v_{tmp})$, $F_{\leq k,e_j}(T_{v_{tmp}}; f, g; v_{tmp}))$ ($j=0,1\dots,k$) by
 \Statex \yy{0.4}setting $ T:=T_{v_{t}}; v_i:=v_{t} (resp.~T:=T_{v_{tmp}}; v_i:=v_{tmp})$ and calling \Statex \yy{0.4}Algorithm~\ref{Algorithm:degreecontibcsubtreecontfixedu};

\State Update $N_T$ with Eq.~(\ref{equ:bcntlongformule}) \;

\EndWhile

\EndProcedure

\end{algorithmic}
\end{algorithm}

{\small\begin{equation}\label{equ:bctwoverticeslongformule}
\begin{split}
N_T=&\prod\limits_{e\in E(P_{v_i v_j})}g(e)\Big(\sum\limits_{m=1}^{k-1}f(v_i)_{o}^{m}\sum\limits_{m=1}^{k-1}f(v_j)_{o}^{m}\prod\limits_{i=1}^{l-1}(\sum\limits_{m=0}^{k-2}f(u_i)_o^m)^{1-i(\text{mod}~2)}(\sum\limits_{m=0}^{k-2}f(u_i)_e^m)^{i(\text{mod}~2)}\\
&+\sum\limits_{m=0}^{k-1}f(v_i)_{e}^{m}\sum\limits_{m=0}^{k-1}f(v_j)_{e}^{m}\prod\limits_{i=1}^{l-1}(\sum\limits_{m=0}^{k-2}f(u_i)_o^m)^{i(\text{mod}~2)}(\sum\limits_{m=0}^{k-2}f(u_i)_e^m)^{1-i(\text{mod}~2)}\Big)
\end{split}
\end{equation}}

{\small \begin{equation}\label{equ:bc2twoverticeslongformule}
\begin{split}
N_T=&\prod\limits_{e\in E(P_{v_i v_j})}g(e)\Big(\sum\limits_{m=1}^{k-1}f(v_i)_{o}^{m}\sum\limits_{m=0}^{k-1}f(v_j)_{e}^{m}\prod\limits_{i=1}^{l-1}(\sum\limits_{m=0}^{k-2}f(u_i)_o^m)^{1-i(\text{mod}~2)}(\sum\limits_{m=0}^{k-2}f(u_i)_e^m)^{i(\text{mod}~2)}\\
&+\sum\limits_{m=0}^{k-1}f(v_i)_{e}^{m}\sum\limits_{m=1}^{k-1}f(v_j)_{o}^{m}\prod\limits_{i=1}^{l-1}(\sum\limits_{m=0}^{k-2}f(u_i)_o^m)^{i(\text{mod}~2)}(\sum\limits_{m=0}^{k-2}f(u_i)_e^m)^{1-i(\text{mod}~2)}\Big)
\end{split}
\end{equation}}
\begin{algorithm}[!ht]
\caption{Generating function $F_{{BC}_{\leq k}}(T;f,g;v_i, v_j)$ for enumerating BC-subtrees containing two distinct vertices $v_i, v_j\in V(T)(i\neq j)$ with maximum degree $\leq k$ of weighted tree $T=(V(T),E(T);f,g)$}
\label{Algorithm:BCsubtreenumaxfixedkconttwovertex}
\begin{algorithmic}[1]
\State Initialize with $(f(v_s)_{o}^{0},f(v_s)_{o}^{1},\dots ,f(v_s)_{o}^{k};f(v_s)_{e}^{0},f(v_s)_{e}^{1},\dots,f(v_s)_{e}^{k}$)($f(v_s)_{o}^{0}=1, f(v_s)_{e}^{0}=y,f(v_s)_{o}^{1}=,\dots,=f(v_s)_{o}^{k}=f(v_s)_{e}^{1}=,\dots,=f(v_s)_{e}^{k}=0$) for each vertex $v_s \in V(T)$;\;

\State Let $T_{tmp}:=T$, and set $N_{T}=0$;\;

\State  \Call{Contract5}{ };\;
\Statex \small{/* denote $P_{v_i v_j}=v_i(u_0)u_1u_2\dots u_{l-1}v_j(u_l)$ the unique path of length $l(\geq1)$ of $T$ connecting $v_i$ and $v_j$, where $v_i=u_0$ and $v_j=u_l$.*/}

\If{$l\equiv 0 (\text{mod 2})$}

\State  Update $N_{T}$ with Eq.~\eqref{equ:bctwoverticeslongformule};\;

\Else
\State  Update $N_{T}$ with Eq.~\eqref{equ:bc2twoverticeslongformule};\;

\EndIf

\State \Return $F_{{BC}_{\leq k}}(T;f,g;v_i, v_j)=N_T$.

\Procedure {Contract5}{ }

\While{$T_{tmp}$ has pendant vertex that is different from $v_i$ and $v_j$}
\State Choose a pendant vertex $u$, which is different from $v_i$ and $v_j$, and  \;

\Statex \yy{0.4} let $e=(u,p)$ denote the pendant edge. \;

\For{$(j=1;j\leq k;j++)$}

\State  Update $f(p)_{o}^{j}$ with $f(p)_{o}^{j}+f(p)_{o}^{j-1} g(e) \sum\limits_{m=0}^{k-1}f(u')_{e}^{m}$;\;

\State Update $f(p)_{e}^{j}$ with $f(p)_{e}^{j}+f(p)_{e}^{j-1}g(e) \sum\limits_{m=1}^{k-1}f(u')_{o}^{m}$;\;
\EndFor

\State Eliminate vertex $u$ and edge $e$ and let $T_{tmp}:=T_{tmp}\backslash((u,p)\cup u)$; \;

\EndWhile

\EndProcedure

\end{algorithmic}
\end{algorithm}

Consequently, with Algorithms \ref{Algorithm:subtreenumaxfixedk}-\ref{Algorithm:BCsubtreenumaxfixedkconttwovertex}, we can obtain the number of various subtrees, BC-subtrees with maximum degree $k$ of trees is just $F_{\leq k}(\cdot)-F_{\leq k-1}(\cdot)$ and  $F_{{BC}_{\leq k}}(\cdot)-F_{{BC}_{\leq k-1}}(\cdot)$, respectively.

\section{Algorithm implementation and data discussion}
\label{sec:algoforimplementanddiscuss}
$\bold{Example~1}$ To better understand the Algorithms \ref{Algorithm:subtreenumaxfixedk}-\ref{Algorithm:subtreenumaxfixedkconttwovertex}, we illustrate the procedures of computing the respective generating functions for a tree $T$ (see Fig.~\ref{fig:illustrationcasedemosubtreemaxk}), and we initialize each vertex weight $(y,0,0,0,0)$ and edge weight $z$ and set $k=4$.

\begin{figure}[!htbp]
\centering
\includegraphics[width=0.93\textwidth]{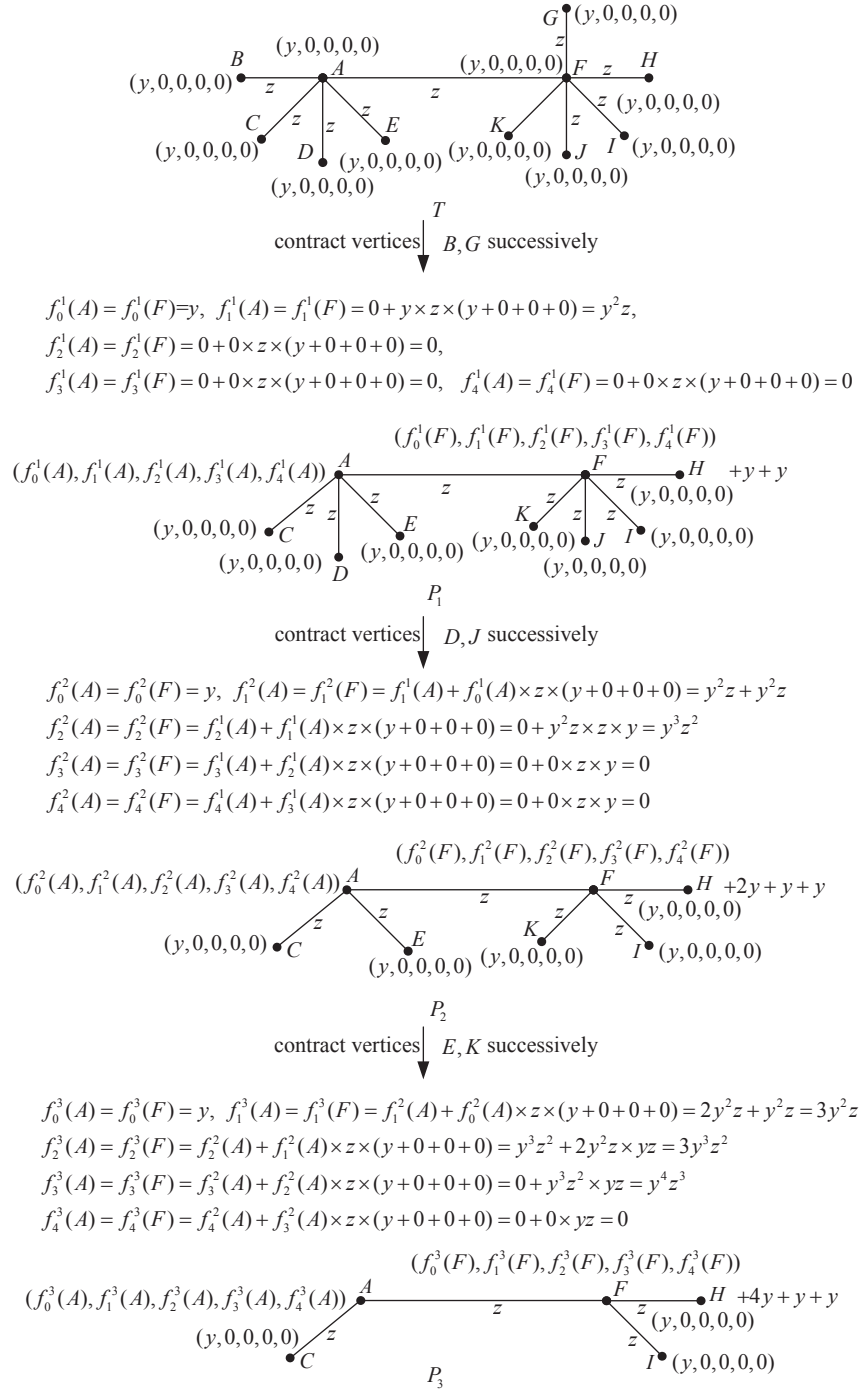}\\
\caption{Illustration of the procedures for computing $F_{\leq 4}(T;(y,0,0,0,0),z)$, $F_{\leq 4}(T;(y,0,0,0,0),z;A)$, $F_{\leq 4}(T;(y,0,0,0,0),z;A, H)$ of a weighted tree $T$ by Algorithms \ref{Algorithm:subtreenumaxfixedk}, \ref{Algorithm:subtreenumaxfixedkcontavertex}, \ref{Algorithm:subtreenumaxfixedkconttwovertex}, respectively.}
\label{fig:illustrationcasedemosubtreemaxk}
\end{figure}

\addtocounter{figure}{-1}
\begin{figure}[!tp]
\centering
\includegraphics[width=\textwidth]{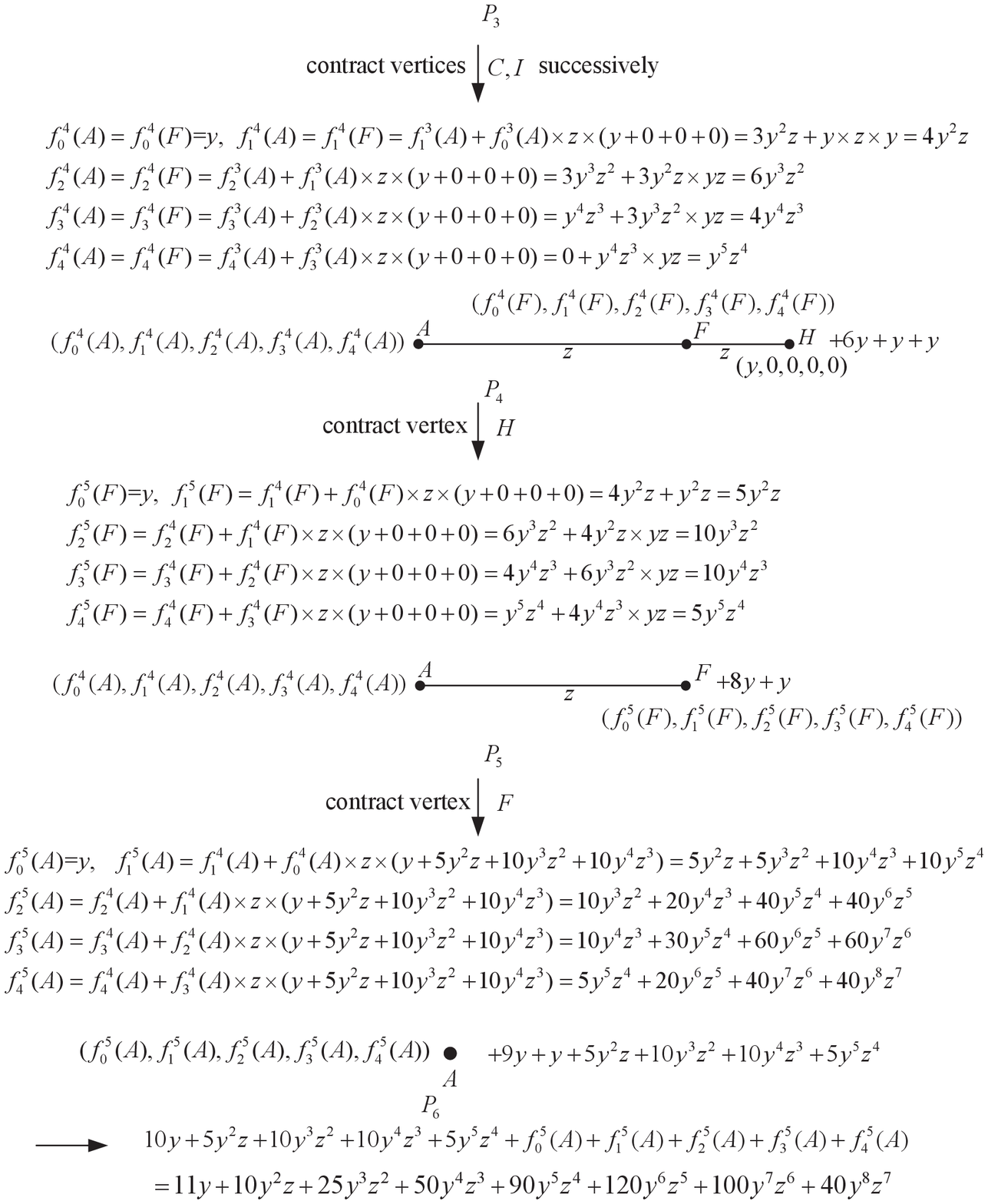}\\
\caption{Illustration of the procedures for computing $F_{\leq 4}(T;(y,0,0,0,0),z)$, $F_{\leq 4}(T;(y,0,0,0,0),z;A)$, $F_{\leq 4}(T;(y,0,0,0,0),z;A,H)$ of a weighted tree $T$ by Algorithms \ref{Algorithm:subtreenumaxfixedk}, \ref{Algorithm:subtreenumaxfixedkcontavertex}, \ref{Algorithm:subtreenumaxfixedkconttwovertex}, respectively(contd.).}
\label{fig:illustrationcasedemosubtreemaxkconti}
\end{figure}

By Algorithm \ref{Algorithm:subtreenumaxfixedkconttwovertex} and phase $P_4$ of Fig.~\ref{fig:illustrationcasedemosubtreemaxk}, we have $F_{\leq 4}(T;(y,0,0,0,0),z;A, H)=(y+4y^2z+6y^3z^2+4y^4z^3)\times z\times(y+4y^2z+6y^3z^2)\times z\times(y+0+0+0)=24y^8z^7 + 52y^7z^6 + 52y^6z^5 + 28y^5z^4 + 8y^4z^3 + y^3z^2$. By Algorithm \ref{Algorithm:subtreenumaxfixedkcontavertex} and phase $P_6$ of Fig.~\ref{fig:illustrationcasedemosubtreemaxk}, we have $F_{\leq 4}(T;(y,0,0,0,0),z;A)=f_0^5(A)+f_1^5(A)+f_2^5(A)+f_3^5(A)+f_4^5(A)=40y^8z^7 + 100y^7z^6 + 120y^6z^5 + 85y^5z^4 + 40y^4z^3 + 15y^3z^2 + 5y^2z + y$. By Algorithm \ref{Algorithm:subtreenumaxfixedk} and phase $P_6$ of Fig.~\ref{fig:illustrationcasedemosubtreemaxk}, we have $F_{\leq 4}(T;(y,0,0,0,0),z)=11y+10y^2z+25y^3z^2+50y^4z^3+90y^5z^4+120y^6z^5+100y^7z^6+40y^8z^7$. Clearly, the coefficients of $y^iz^{i-1}(i\geq 1)$ represents the number of subtrees on $i$ vertices with maximum degree $\leq 4$ of $T$. For instance, there are 11, 10, 25, 50, 90 120, 100, 40 subtrees on 1 to 8 vertices with maximum degree $\leq 4$ of $T$, respectively.

Moreover, by substituting $y=1, z=1$ to the above generating functions $F_{\leq 4}(T;(y,0,0,0,0),z)$, $F_{\leq 4}(T;(y,0,0,0,0),z;A)$, $F_{\leq 4}(T;(y,0,0,0,0),z;A, H)$, we have $\eta_{\leq 4}(T;A,H)=165$, $\eta_{\leq 4}(T;A)=406$, $\eta_{\leq 4}(T)=446$. Namely, there are 446 subtrees in total, 406 subtrees that contain vertex $A$, and 165 subtrees that contain vertices $A$ and $H$ of $T$, with maximum degree $\leq 4$.

$\bold{Example~2}$ Similarly, we illustrate the procedures of Algorithms \ref{Algorithm:degreecontibcsubtreecontfixedu} and \ref{Algorithm:BCsubtreenumaxfixedkcontavertex} to compute the generating functions $F_{\leq k,o_j}(T;f,g;v_{t}),F_{\leq k,e_j}(T;f,g;v_{t})$ ($j=0,1\dots,k$)(see Fig.~\ref{fig:oddevenandtwoforbc}), and $F_{{BC}_{\leq k}}(T;f,g;v_t)$ for a fixed vertex $v_{t} \in V(T)$ (see Fig.~\ref{fig:bcfixverillu}), and also the procedures of Algorithms \ref{Algorithm:degereekbcsubtreenum} and \ref{Algorithm:BCsubtreenumaxfixedkconttwovertex} to compute $F_{{BC}_{\leq k}}(T;f,g)$ (see Fig.~\ref{fig:bcmaxk}), $F_{{BC}_{\leq k}}(T;f,g;v_i, v_j)$ containing two distinct vertices $v_i, v_j\in V(T)\\(i\neq j)$ of a tree $T$(see Fig.~\ref{fig:oddevenandtwoforbc}). Here, we set $k=3$ and initialize its each vertex weight with $(1,0,0,0;y,0,0,0)$ and edge weight with $z$.

\begin{figure}[!htbp]
\centering
\includegraphics[width=0.96\textwidth]{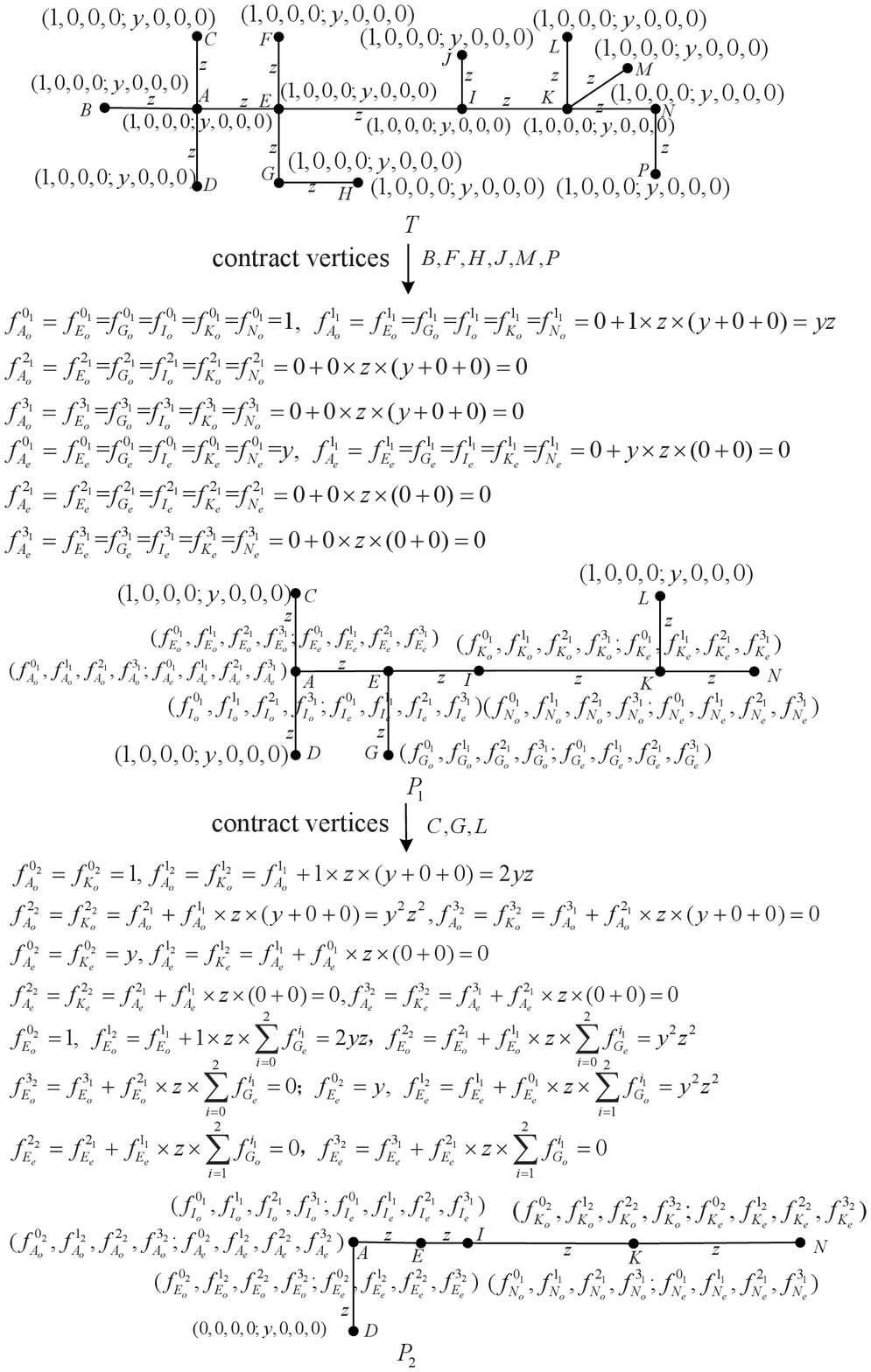}\\
\caption{Illustration of the procedures for computing $F_{\leq 3,o_j}(T;(1,0,0,0;y,0,0,0),z;I),F_{\leq 3,e_j}(T;(1,0,0,0;y,0,0,0),z;I)$ ($j=0,1,2,3$), $F_{{BC}_{\leq 3}}(T;(1,0,0,0;y,0,0,0),z;A, K)$ of a weighted tree $T$ by Algorithms \ref{Algorithm:degreecontibcsubtreecontfixedu}, \ref{Algorithm:BCsubtreenumaxfixedkconttwovertex} respectively.}
\label{fig:oddevenandtwoforbc}
\end{figure}
\addtocounter{figure}{-1}
\begin{figure}[!htbp]
\centering
\includegraphics[width=\textwidth]{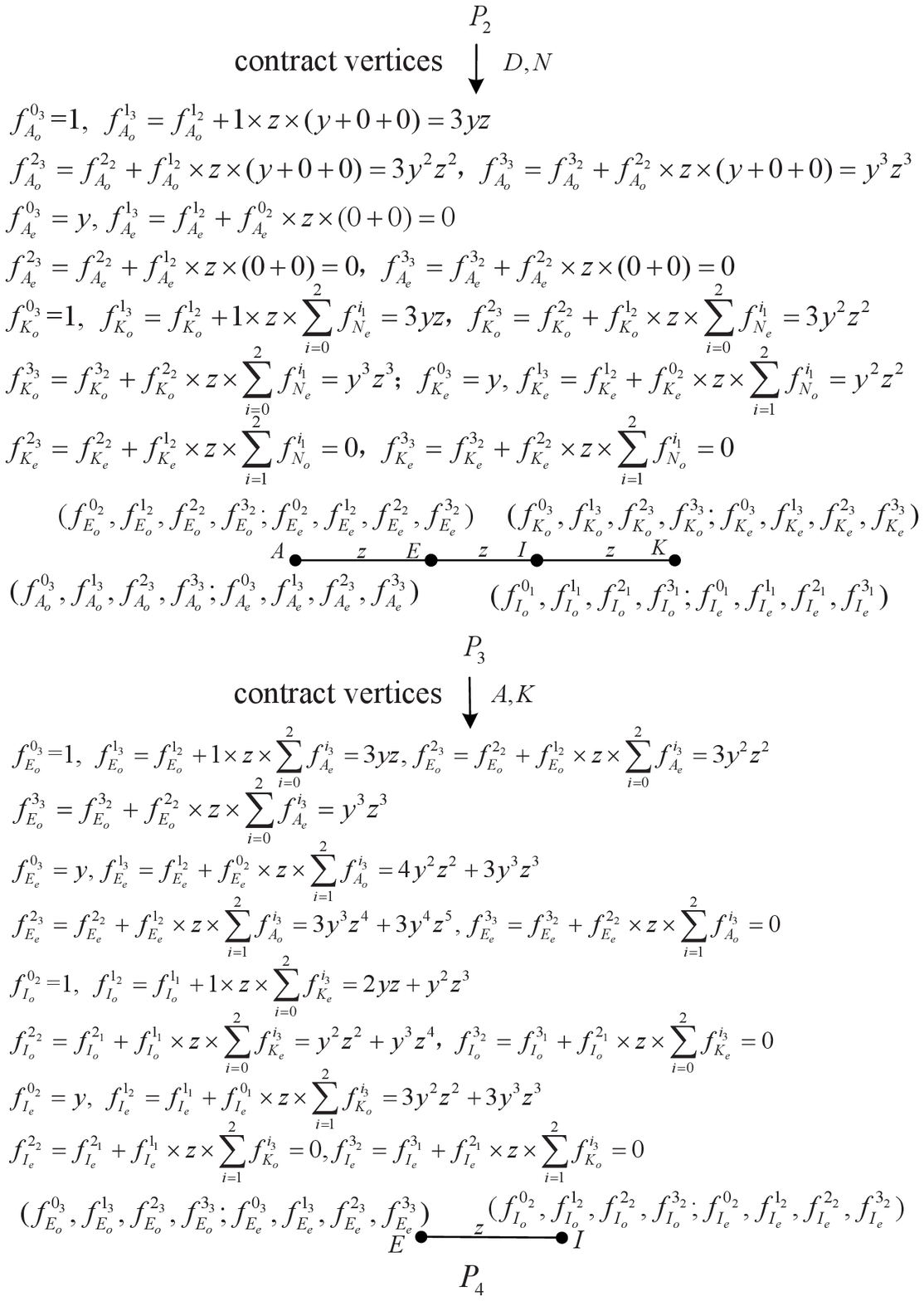}\\
\caption{Illustration of the procedures for computing $F_{\leq 3,o_j}(T;(1,0,0,0;y,0,0,0),z;I),F_{\leq 3,e_j}(T;(1,0,0,0;y,0,0,0),z;I)$ ($j=0,1,2,3$), $F_{{BC}_{\leq 3}}(T;(1,0,0,0;y,0,0,0),z;A, K)$ of a weighted tree $T$ by Algorithms \ref{Algorithm:degreecontibcsubtreecontfixedu}, \ref{Algorithm:BCsubtreenumaxfixedkconttwovertex} respectively (contd.).}
\end{figure}

\addtocounter{figure}{-1}
\begin{figure}[!htbp]
\centering
\includegraphics[width=\textwidth]{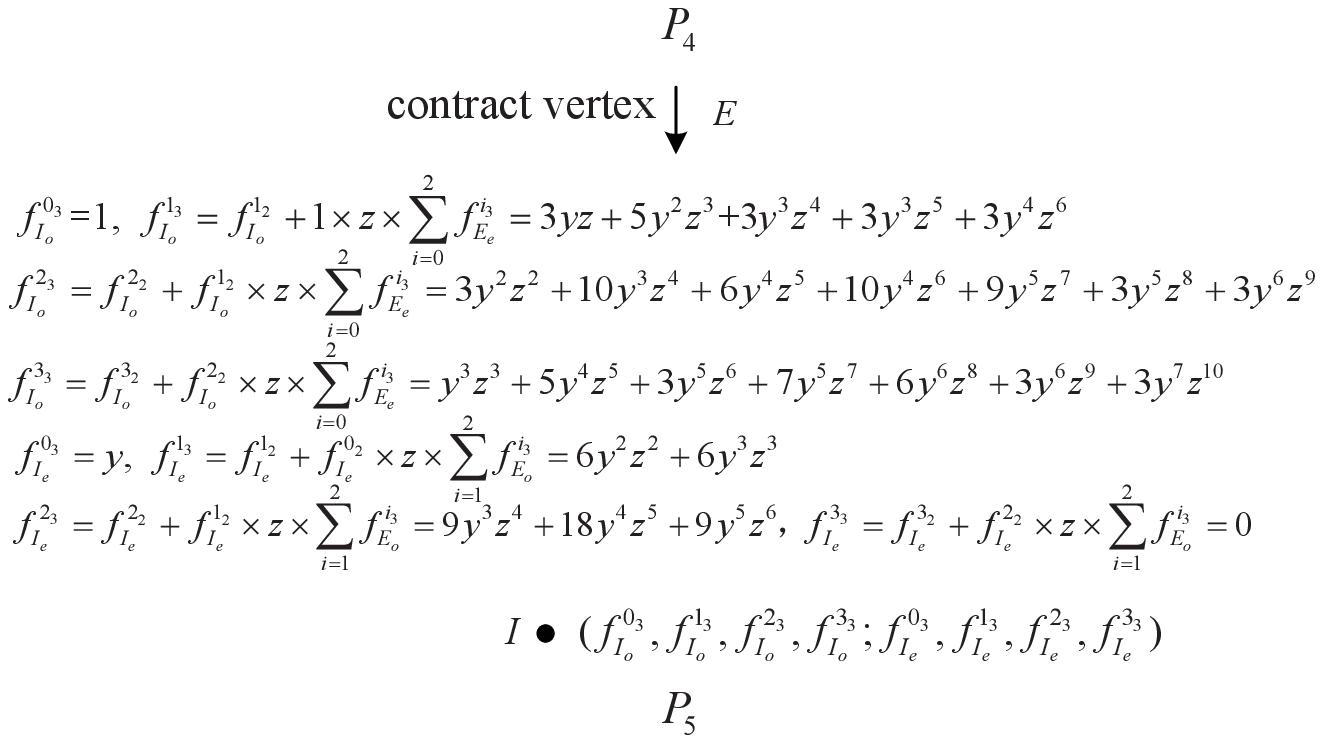}\\
\caption{Illustration of the procedures for computing $F_{\leq 3,o_j}(T;(1,0,0,0;y,0,0,0),z;I),F_{\leq 3,e_j}(T;(1,0,0,0;y,0,0,0),z;I)$ ($j=0,1,2,3$), $F_{{BC}_{\leq 3}}(T;(1,0,0,0;y,0,0,0),z;A, K)$ of a weighted tree $T$ by Algorithms \ref{Algorithm:degreecontibcsubtreecontfixedu}, \ref{Algorithm:BCsubtreenumaxfixedkconttwovertex} respectively (contd.).}
\end{figure}

\begin{figure}[!htbp]
\centering
\includegraphics[width=1.05\textwidth]{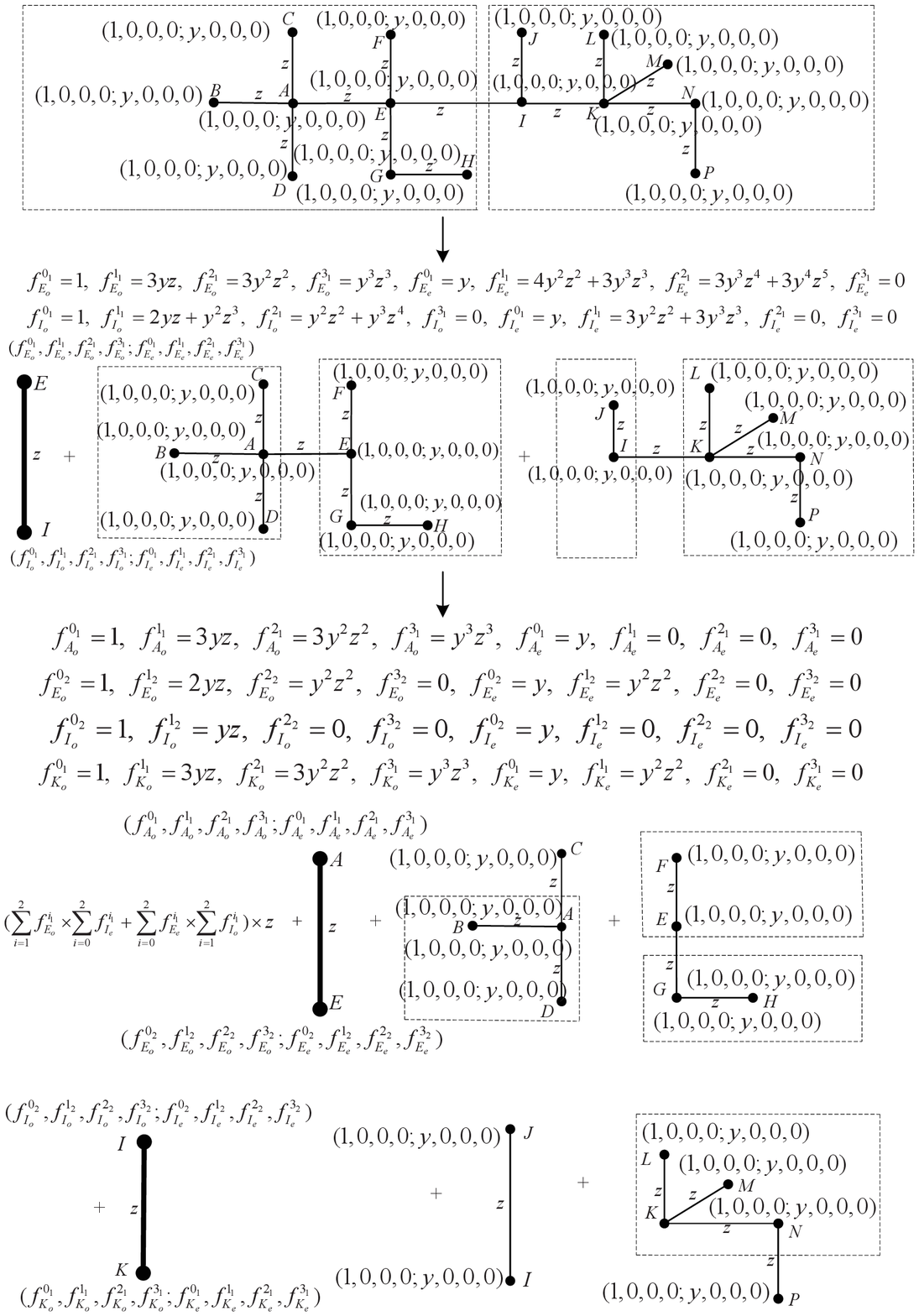}\\
\caption{Illustration of the procedures for computing $F_{{BC}_{\leq 3}}(T;(1,0,0,0;y,0,0,0),z)$ of a weighted tree $T$ by Algorithm \ref{Algorithm:degereekbcsubtreenum}.}
\label{fig:bcmaxk}
\end{figure}
\addtocounter{figure}{-1}
\begin{figure}[!htbp]
\centering
\includegraphics[width=\textwidth]{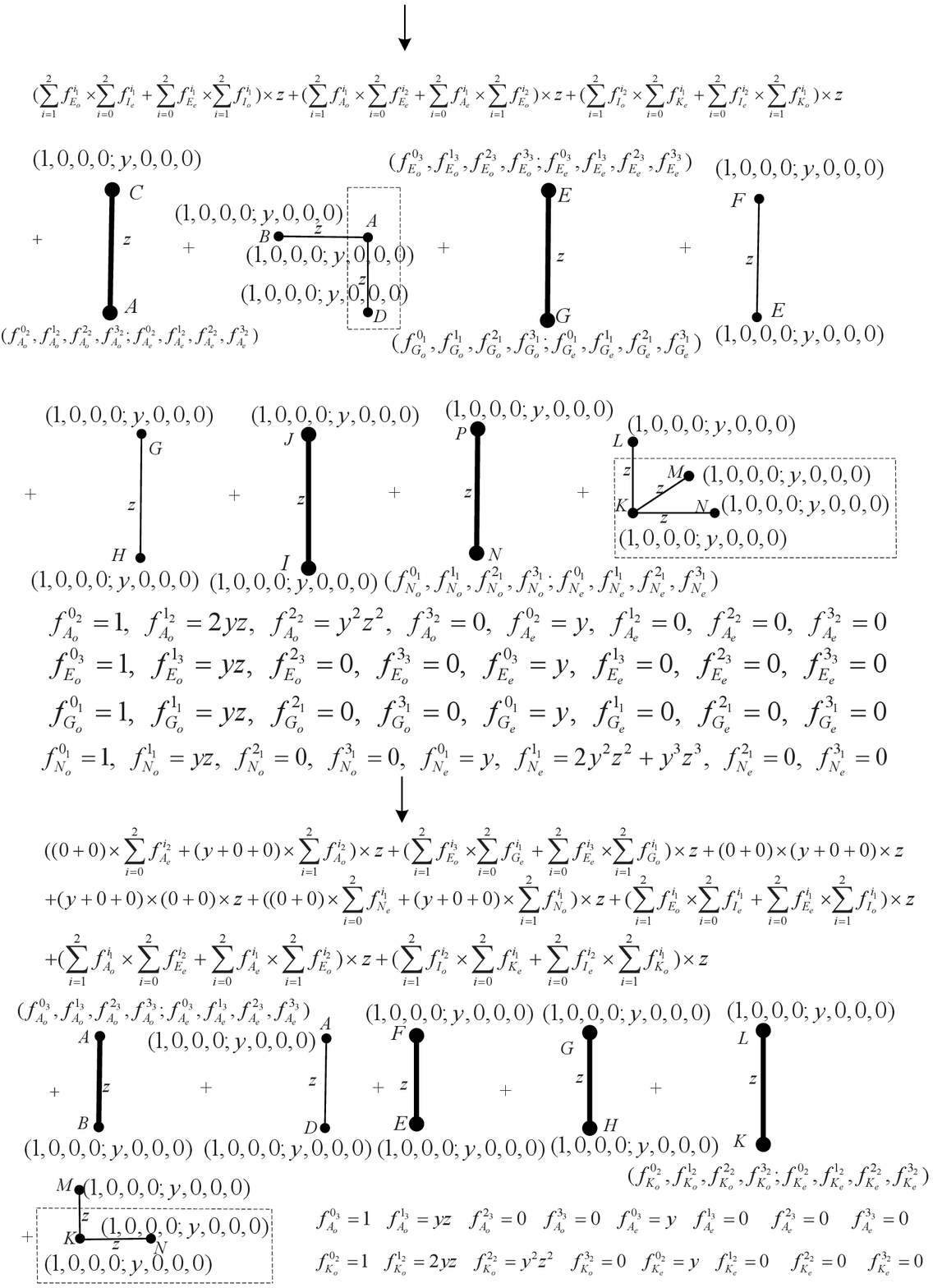}\\
\caption{Illustration of the procedures for computing $F_{{BC}_{\leq 3}}(T;(1,0,0,0;y,0,0,0),z)$ of a weighted tree $T$ by Algorithm \ref{Algorithm:degereekbcsubtreenum} (contd.).}
\label{fig:bcillustrationcasedemosubtreemaxk}
\end{figure}
\addtocounter{figure}{-1}
\begin{figure}[!htbp]
\centering
\includegraphics[width=\textwidth]{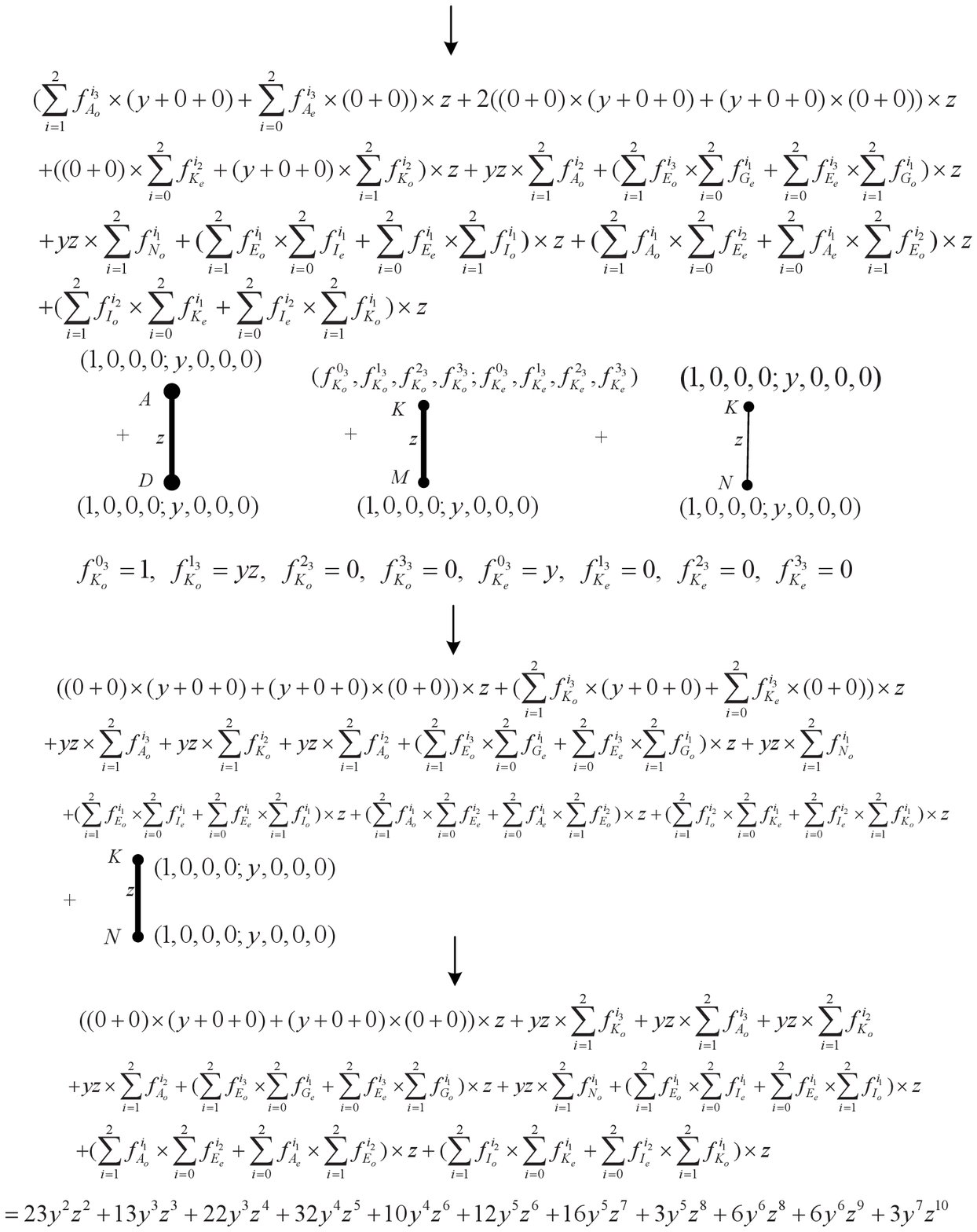}\\
\caption{Illustration of the procedures for computing $F_{{BC}_{\leq 3}}(T;(1,0,0,0;y,0,0,0),z)$ of a weighted tree $T$ by Algorithm \ref{Algorithm:degereekbcsubtreenum} (contd.).}
\end{figure}

\begin{figure}[!htbp]
\centering
\includegraphics[width=\textwidth]{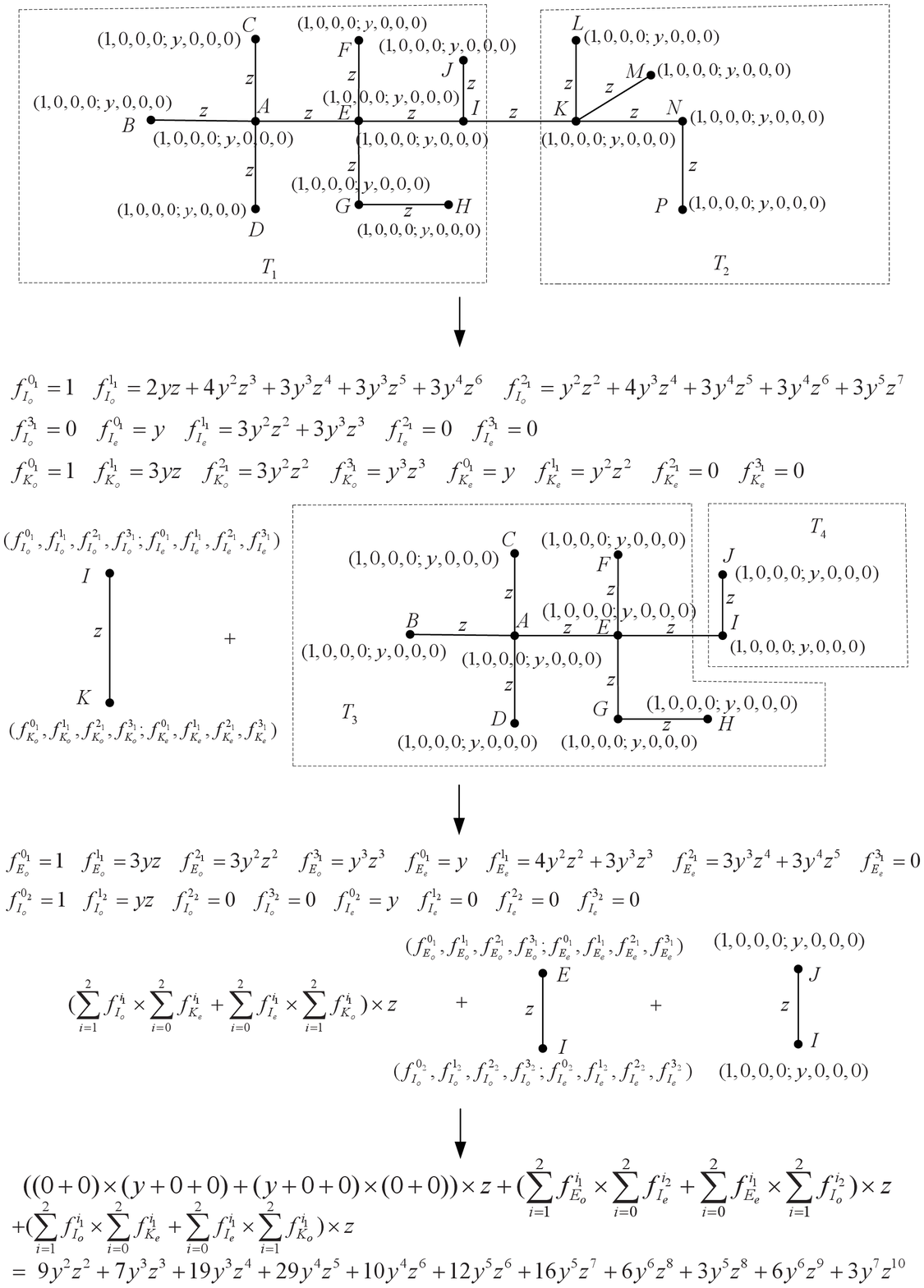}\\
\caption{Illustration of the procedures for computing $F_{{BC}_{\leq 3}}(T;(1,0,0,0;y,0,0,0),z;I)$ of a weighted tree $T$ by Algorithm \ref{Algorithm:BCsubtreenumaxfixedkcontavertex}.}
\label{fig:bcfixverillu}
\end{figure}

By Algorithm \ref{Algorithm:degreecontibcsubtreecontfixedu}, phase $P_5$ of Fig.~\ref{fig:oddevenandtwoforbc}, we have the $F_{\leq 3,o_0}(T;(1,0,0,0;y,0,0,0),\\z;I)=1, F_{\leq 3,e_0}(T;(1,0,0,0;y,0,0,0),z;I)=y$, $F_{\leq 3,o_1}(T;(1,0,0,0;y,0,0,0),z;\\I)=3yz+5y^2z^3+3y^3z^4+3y^3z^5+3y^4z^6, F_{\leq 3,e_1}(T;(1,0,0,0;y,0,0,0),z;I)=6y^2z^2+6y^3z^3$, $F_{\leq 3,o_2}(T;(1,0,0,0;y,0,0,0),z;I)=3y^2z^2+10y^3z^4+6y^4z^5+10y^4z^6+9y^5z^7+3y^5z^8+3y^6z^9, F_{\leq 3,e_2}(T;(1,0,0,0;y,0,0,0),z;I)=9y^3z^4+18y^4z^5+9y^5z^6$, $F_{\leq 3,o_3}(T;(1,0,0,0;y,0,0,0),z;I)=y^3z^3+5y^4z^5+3y^5z^6+7y^5z^7+6y^6z^8+3y^6z^9+3y^7z^{10}, F_{\leq 3,e_3}(T;(1,0,0,0;y,0,0,0),z;I)=0$.

By Algorithm \ref{Algorithm:degereekbcsubtreenum} and last phase of Fig.~\ref{fig:bcillustrationcasedemosubtreemaxk}, we have $F_{{BC}_{\leq 3}}(T;(1,0,0,0;y,0,0,0),\\z)=23y^2z^2+13y^3z^3+22y^3z^4+32y^4z^5+10y^4z^6+12y^5z^6+16y^5z^7+3y^5z^8+6y^6z^8+6y^6z^9+3y^7z^{10}$. By  Algorithm \ref{Algorithm:BCsubtreenumaxfixedkcontavertex} and Fig.~\ref{fig:bcfixverillu}, we have $F_{{BC}_{\leq 3}}(T;(1,0,0,0;y,\\0,0,0),z;I)=9y^2z^2+7y^3z^3+19y^3z^4+29y^4z^5+10y^4z^6+12y^5z^6+16y^5z^7+6y^6z^8+3y^5z^8+6y^6z^9+3y^7z^{10}$. By Algorithm \ref{Algorithm:BCsubtreenumaxfixedkconttwovertex} and phase $P_3$ of Fig.~\ref{fig:oddevenandtwoforbc}, we have $F_{{BC}_{\leq 3}}(T;(1,0,0,0;y,0,0,0),z;A, K)=(3yz+3y^2z^2)\times(y+y^2z^2+0)\times(y+0)\times(1+yz)\times z^3+(y+0+0)\times(3yz+3y^2z^2)\times(1+2yz)\times(y+0)\times z^3=3y^6z^8 + 6y^5z^7 + 9y^5z^6 + 3y^4z^6 + 15y^4z^5 + 6y^3z^4$.

By substituting $y=1, z=1$ to the above generating functions, we can obtain $\eta_{\leq 3,o_1}(T;I)=17, \eta_{\leq 3,o_2}(T;I)=44, \eta_{\leq 3,o_3}(T;I)=28$, $\eta_{\leq 3,e_0}(T;I)=1, \eta_{\leq 3,e_1}(T;I)=12, \eta_{\leq 3,e_2}(T;I)=36,\eta_{\leq 3,e_3}(T;I)=0$. $\eta_{{BC}_{\leq 3}}(T)=146$, $\eta_{{BC}_{\leq 3}}(T;I)=120$, $\eta_{{BC}_{\leq 3}}(T;A,K)=42$, where $\eta_{\leq 3,o_j}(T;I)(j=1,2,3)$ is the number of subtrees in $\mathcal{S}_{\leq 3,o_j}(T;I)$ and $\eta_{\leq 3,e_j}(T;I)(j=0,1,2,3)$ is the number of subtrees in $\mathcal{S}_{\leq 3,e_j}(T;I)$.

Moreover, by substituting $y = 1$ to generation function $F_{{BC}_{\leq 3}}(T;(1,0,0,0;y,\\0,0,0),z)$, we can obtain the edge generating function $F_{{BC}_{\leq 3}}(T;(1,0,0,0;1,0,0,0),\\z)=23z^2+13z^3+22z^4+32z^5+22z^6+16z^7+9z^8+6z^9+3z^{10}$, from the coefficients, we know that there are 23, 13, 22, 32, 22, 16, 9,  6, 3 BC-subtrees on 3 to 11 vertices, respectively,  with maximum degree $\leq 3$.

Another interesting question is: what proportion of all the subtrees(resp.~ BC-subtrees) are subtrees (resp.~ BC-subtrees) with maximum degree $\leq k(\geq 2)$ ?

We now define ratio $r_{k}=\frac{\tilde{\eta}_{\leq k}(T)}{\tilde{\eta}(T)}$(resp.~$\tilde{r}_{k}=\frac{\tilde{\eta}_{{BC}_{\leq k}}(T)}{\tilde{\eta}_{BC}(T)}$), where $\tilde{\eta}(T)$ (resp.~$\tilde{\eta}_{BC}(T)$) and $\tilde{\eta}_{\leq k}(T)$ (resp.~$\tilde{\eta}_{{BC}_{\leq k}}(T)$) represent the average subtree (resp.~BC-subtree) number and average number of subtrees (resp.~BC-subtrees) with maximum degree $\leq k$ of $T$, respectively, here $\tilde{\eta}(T)$ (resp.~$\tilde{\eta}_{BC}(T)$) could be obtained by using the algorithm presented in reference \cite{yan06} (resp.~\cite{yang2013bc}).

\begin{figure}[!hp]
  \centering
  \noindent
  \subfigure[The proportion of subtrees with maximum degree $\leq k$ to all subtrees of 3000 randomly generated trees on $n$ vertices.]{
    \label{fig:stdensitiesdendrimer:a}
    \includegraphics[width=0.92\textwidth]{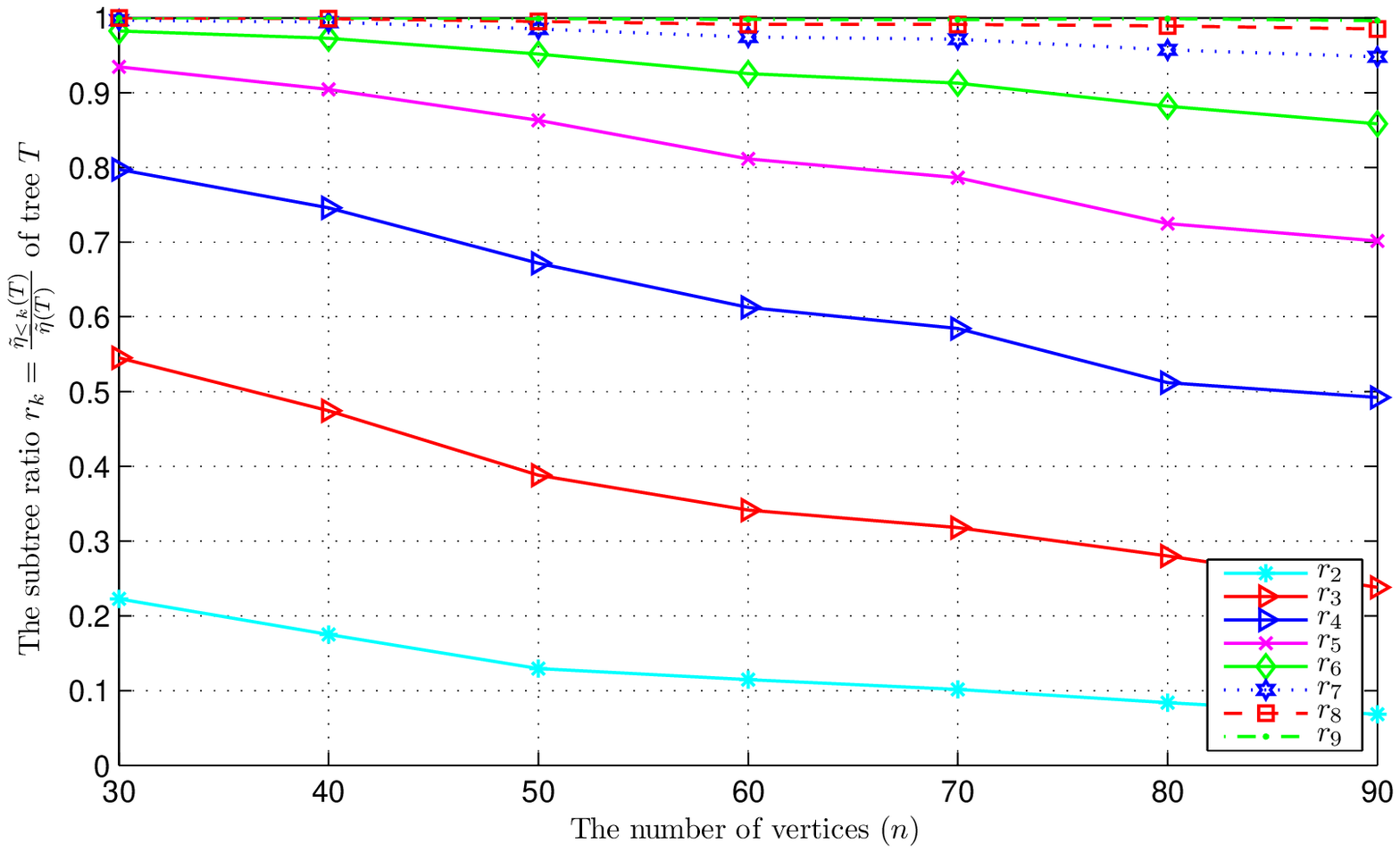}}

    \subfigure[The proportion of BC-subtrees with maximum degree $\leq k$ to all BC-subtrees of 3000 randomly generated trees on $n$ vertices.]{
    \label{fig:stdensitiesdendrimer:b}
  \includegraphics[width=0.92\textwidth]{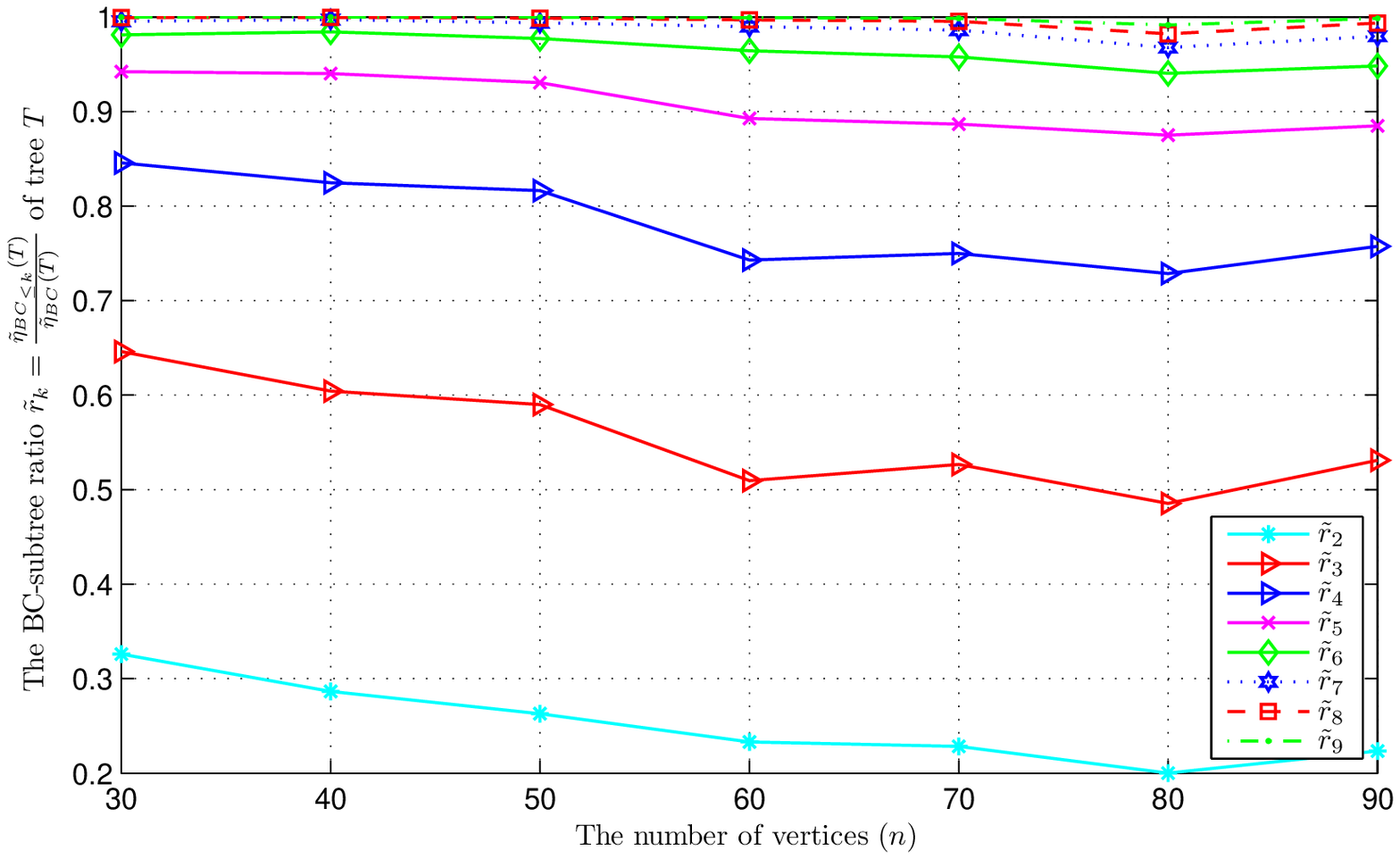}}

  \caption{Asymptotic proportion $r_{k}$ (resp.~$\tilde{r}_{k}$) of subtrees (resp.~BC-subtrees) with maximum degree $\leq k$ to all subtrees (resp.~BC-subtrees) of 3000 randomly generated trees on $n=30,40,50,60,70,80,90$ vertices.}
  \label{fig:propofsubandbcsubtrees} 
\end{figure}

Through 3000 randomly generated trees $T_i(i=1,2\dots,3000)$ on 30, 40, 50, 60, 70, 80, 90 vertices, with algorithms \ref{Algorithm:subtreenumaxfixedk} and \ref{Algorithm:degereekbcsubtreenum}, we can obtain the asymptotic ratios of trees on $n (= 30, 40, 50, 60, 70, 80, 90)$ vertices,  by observing the Fig. \ref{fig:propofsubandbcsubtrees}, we know that the maximum degree of almost all subtrees (resp.~BC-subtrees) are not bigger than 8. Namely, most of the subtree are subtrees with small degree ($\leq 8$).

\section{Concluding remarks}
\label{Sec:Conclusion}
In this paper, by way of generating functions, we presented algorithms of enumerating all subtrees, subtrees containing a fixed vertex, subtrees containing two distinct vertices, with maximum degree $\leq k$ of trees in  Section \ref{sec:algoforsubtree}, and algorithms of enumerating all BC-subtrees, BC-subtrees containing a fixed vertex, BC-subtrees containing two distinct vertices, with maximum degree $\leq k$ of trees in Section \ref{sec:algoforbcsubtree}. Section \ref{sec:algoforimplementanddiscuss} illustrates the procedures of the proposed enumerating algorithms and briefly discusses the ratios of subtrees (resp.~ BC-subtrees) with maximum degree $k(\geq 2)$ to all subtrees in general tree.

These studies further explored the behavior of subtree numbers and BC-subtree numbers in graphs. In particular, this seems to be the first time that degree constraints were put on BC-subtrees.

For future works, we plan to study the enumeration problem of subtrees and BC-subtrees under other special constraints, such as subtrees and BC-subtrees with diameter $\leq d(\geq 2)$ of trees.

\section*{Acknowledgment}
This work is supported by
the National Natural Science Foundation of China (grant nos. 61702291, 61772102, 11971311, 11531001); Program for Science \& Technology Innovation Talents in Universities of Henan Province(grant no. 19HASTIT029).

\bibliographystyle{elsarticle-num}
\bibliography{NSPHC_subtree}

\begin{thebibliography}{10}
\expandafter\ifx\csname url\endcsname\relax
  \def\url#1{\texttt{#1}}\fi
\expandafter\ifx\csname urlprefix\endcsname\relax\def\urlprefix{URL }\fi
\expandafter\ifx\csname href\endcsname\relax
  \def\href#1#2{#2} \def\path#1{#1}\fi

\bibitem{xiao2017trees}
Y.~Xiao, H.~Zhao, Z.~Liu, Y.~Mao, Trees with large numbers of subtrees,
  International Journal of Computer Mathematics 94~(2) (2017) 372--385.

\bibitem{knudsen2003optimal}
B.~Knudsen, Optimal multiple parsimony alignment with affine gap cost using a
  phylogenetic tree, in: Algorithms in Bioinformatics, Springer, 2003, pp.
  433--446.

\bibitem{jam1983average}
R.~E. Jamison, On the average number of nodes in a subtree of a tree, Journal
  of Combinatorial Theory, Series B 35~(3) (1983) 207--223.

\bibitem{1984Monotonicity}
R.~E. Jamison, Monotonicity of the mean order of subtrees, Journal of
  Combinatorial Theory, Series B 37~(1) (1984) 70--78.

\bibitem{merrifield1989topological}
R.~E. Merrifield, H.~E. Simmons, Topological Methods in Chemistry, Wiley, New
  York, 1989.

\bibitem{Jamison1990Alternating}
R.~E. Jamison, Alternating whitney sums and matchings in trees, part ii,
  Discrete Mathematics 79~(2) (1990) 177--189.

\bibitem{haslegrave2014extremal}
J.~Haslegrave, Extremal results on average subtree density of series-reduced
  trees, Journal of Combinatorial Theory, Series B 107 (2014) 26--41.

\bibitem{yang2013bc}
Y.~Yang, H.~Liu, H.~Wang, M.~Scott, Enumeration of {BC}-subtrees of trees,
  Theoretical Computer Science, 580 (2015) 59--74.

\bibitem{yang2015subtrees}
Y.~Yang, H.~Liu, H.~Wang, H.~Fu, Subtrees of spiro and polyphenyl hexagonal
  chains, Applied Mathematics and Computation 268 (2015) 547--560.

\bibitem{yan06}
W.~Yan, Y.~Yeh, Enumeration of subtrees of trees, Theoretical Computer Science
  369~(1) (2006) 256--268.

\bibitem{yang2015algorithms}
Y.~Yang, H.~Liu, H.~Wang, S.~Feng, On algorithms for enumerating bc-subtrees of
  unicyclic and edge-disjoint bicyclic graphs, Discrete Applied Mathematics 203
  (2016) 184--203.

\bibitem{chin2018subtrees}
A.~J. Chin, G.~Gordon, K.~J. MacPhee, C.~Vincent, Subtrees of graphs, Journal
  of Graph Theory 89 (2018) 413--438.

\bibitem{yang2017spiro}
Y.~Yang, H.~Liu, H.~Wang, S.~Sun, On spiro and polyphenyl hexagonal chains with
  respect to the number of {BC-subtrees}, International Journal of Computer
  Mathematics 94~(4) (2017) 774--799.

\bibitem{deng2012wiener}
H.~Deng, Wiener indices of spiro and polyphenyl hexagonal chains, Mathematical
  and Computer Modelling 55~(3) (2012) 634--644.

\bibitem{dovslic2010chain}
T.~Do{\v{s}}li{\'c}, F.~M{\aa}l{\o}y, Chain hexagonal cacti: {Matchings} and
  independent sets, Discrete Mathematics 310~(12) (2010) 1676--1690.

\bibitem{chen2009six}
X.~Chen, B.~Zhao, P.~Zhao, Six-membered ring spiro chains with extremal
  {Merrifield-Simmons} index and {Hosoya} index, MATCH Communications in
  Mathematical and in Computer Chemistry 62~(3) (2009) 657.

\bibitem{li2012hosoya}
X.~Li, G.~Wang, H.~Bian, R.~Hu, The hosoya polynomial decomposition for
  polyphenyl chains, MATCH Communications in Mathematical and in Computer
  Chemistry 67~(2) (2012) 357.

\bibitem{yang2017algorithms}
Y.~Yang, H.~Liu, H.~Wang, A.~Deng, C.~Magnant, On algorithms for enumerating
  subtrees of hexagonal and phenylene chains, The Computer Journal 60~(5)
  (2017) 690--710.

\bibitem{gutman1987wiener}
I.~Gutman, Wiener numbers of benzenoid hydrocarbons: two theorems, Chemical
  Physics Letters 136~(2) (1987) 134--136.

\bibitem{yangyu2020}
Y.~Yang, X.~Sun, J.~Cao, H.~Wang, X.~Zhang, The expected subtree number index
  in random polyphenylene and spiro chains, Discrete Applied Mathematics 285
  (2020) 483--492.

\bibitem{yang2020trialgorithms}
Y.~Yang, B.~Chen, G.~Zhang, Y.~Li, D.~Sun, H.~Liu, Algorithms based on path
  contraction carrying weights for enumerating subtrees of tricyclic graphs,
  The Computer Journal (2020) doi: 10.1093/comjnl/bxaa084.

\bibitem{sze05}
L.~A. Sz{\'e}kely, H.~Wang, On subtrees of trees, Advances in Applied
  Mathematics 34~(1) (2005) 138--155.

\bibitem{szekely2007binary}
L.~A. Sz{\'e}kely, H.~Wang, Binary trees with the largest number of subtrees,
  Discrete Applied Mathematics 155~(3) (2007) 374--385.

\bibitem{zhang2015minimal}
X.~Zhang, X.~Zhang, The minimal number of subtrees with a given degree
  sequence, Graphs and Combinatorics 31~(1) (2015) 309--318.

\bibitem{kirk2008largest}
R.~Kirk, H.~Wang, Largest number of subtrees of trees with a given maximum
  degree, SIAM Journal on Discrete Mathematics 22~(3) (2008) 985--995.

\bibitem{zhang2013number}
X.~Zhang, X.~Zhang, D.~Gray, H.~Wang, The number of subtrees of trees with
  given degree sequence, Journal of Graph Theory 73~(3) (2013) 280--295.

\bibitem{vince2010average}
A.~Vince, H.~Wang, The average order of a subtree of a tree, Journal of
  Combinatorial Theory, Series B 100~(2) (2010) 161--170.

\bibitem{szekely2014extremal}
L.~A. Sz{\'e}kely, H.~Wang, Extremal values of ratios: Distance problems vs.
  subtree problems in trees {II}, Discrete Mathematics 322 (2014) 36--47.

\bibitem{wagner2007correlation}
S.~Wagner, Correlation of graph-theoretical indices, SIAM Journal on Discrete
  Mathematics 21~(1) (2007) 33--46.

\bibitem{2019Some}
S.~Li, H.~Wang, S.~Wang, Some extremal ratios of the distance and subtree
  problems in binary trees, Applied Mathematics and Computation 361 (2019)
  232--245.

\bibitem{AKUTSU20152}
T.~Akutsu, T.~Tamura, A.~A. Melkman, A.~Takasu, On the complexity of finding a
  largest common subtree of bounded degree, Theoretical Computer Science 590
  (2015) 2 -- 16.

\bibitem{wangkai20095312583205}
K.~Wang, Z.~Ming, T.-S. Chua, A syntactic tree matching approach to finding
  similar questions in community-based qa services, Boston, MA, United states,
  2009, pp. 187 -- 194.

\bibitem{MilanoSC06}
D.~Milano, M.~Scannapieco, T.~Catarci, Structure-aware {XML} object
  identification, {IEEE} Data Engineering Bulletin 29~(2) (2006) 67--74.

\bibitem{LAU2006385}
H.~C. Lau, T.~H. Ngo, B.~N. Nguyen, Finding a length-constrained maximum-sum or
  maximum-density subtree and its application to logistics, Discrete
  Optimization 3~(4) (2006) 385 -- 391.

\end{thebibliography}

\end{document}